\documentclass[12pt, oneside,reqno]{amsart}
\pdfoutput=1

\usepackage[a4paper,margin=1in]{geometry}
\usepackage{amsmath,amssymb,amsthm,mathtools}
\usepackage{graphicx}
\usepackage{subcaption}
\usepackage{float}
\usepackage{tikz}

\usepackage{appendix}

\usepackage{bbold}

\usepackage[normalem]{ulem}

\usepackage{chngcntr}
\usepackage{longtable}

\usepackage{comment}

\usepackage{hyperref}
\hypersetup{colorlinks,urlcolor=blue,citecolor=blue}

\usepackage[nameinlink,noabbrev]{cleveref}
\numberwithin{figure}{subsection}

\DeclareUnicodeCharacter{02BC}{'}

\newcommand{\C}{\mathbb{C}}
\newcommand{\Log}{\operatorname{Log}} 
\newcommand{\R}{\mathbb{R}}
\newcommand{\Z}{\mathbb{Z}}
\newcommand{\N}{\mathbb{N}}

\theoremstyle{plain}
\newtheorem{thm}{Theorem}[section]     
\newtheorem{lemma}[thm]{Lemma}
\newtheorem{prop}[thm]{Proposition}
\newtheorem{corollary}[thm]{Corollary}

\newtheorem{ex}[thm]{Example}
\theoremstyle{definition}

\theoremstyle{definition}
\newtheorem{defn}[thm]{Definition}
\newtheorem{remark}[thm]{Remark}

\author{Ali Saraeb$^{*}$}

\address{$^{*}$Mathematics Department, The Ohio State University, 231 W. 18th Avenue, Columbus, Ohio 43210, USA}
\email{saraeb.1@osu.edu}

\title{Unimodular Fake M\"obius Functions}

\date{}

\begin{document}

\begin{abstract}
Let $\mathbb{S}^1$ denote the unit circle. We introduce and develop the analytic and \emph{bias} theory of \emph{unimodular fake M\"obius functions}, i.e. multiplicative functions $\mathfrak{f}:\N\to\mathbb{S}^1\cup\{0\}$ whose prime-power values are prescribed by a fixed sequence $\{\varepsilon_k\}_{k\ge1}$ via the rule $\mathfrak{f}(p^k)=\varepsilon_k$ for every prime $p$ and every $k\ge1$.

A key feature of these functions is that their Dirichlet series admit a factorization into complex powers of the Riemann zeta function. Our main analytic result is an explicit formula for the smoothed summatory function $\sum_{n\ge1}\mathfrak{f}(n)e^{-n/x}$, consisting of a leading main term together with a sequence of lower-order terms. The formula may be viewed as an extension of the Selberg-Delange method and is expected to be of independent interest.

As an application, we introduce a notion of bias at a \emph{natural} scale and obtain an explicit criterion distinguishing \emph{persistent} bias, \emph{apparent} bias, and no bias for unimodular fake M\"obius functions.

\end{abstract}


\maketitle
\section{Introduction and Main Results} \label{sec:introduction}

\subsection{Motivation} Many probabilistic and number theoretic quantities of interest exhibit large fluctuations, and a bound on their size tells only part of the story. A natural next question is whether, after the correct normalization, the fluctuations remain centered or exhibit a systematic offset. When such an offset persists in the limit, it is sometimes referred to in the statistics literature as an \emph{asymptotic bias} \cite[Sec. 2.5.2]{shao1999mathematical}.

This is illustrated by the central limit theorem (CLT). Let $X_1,X_2,\ldots$ be i.i.d.\ random variables with mean $\mu$ and variance $\sigma^2$, and let $\overline X_n =n^{-1}\sum_{j\le n} X_j$ denote the sample mean. The CLT states that
$\sqrt n(\overline X_n-\mu) \to \mathcal N(0,\sigma^2)$ as $n\to\infty,$ where the convergence is in distribution. At the scale $\sqrt{n}$, the fluctuations are of order one, and there is no deterministic offset left over. In this sense, $\overline X_n$ is said to be asymptotically unbiased at its \emph{natural} scale $\sqrt{n}$. 

One way a nonzero asymptotic bias $b$ can arise is when one adds a deterministic correction of the same order as the fluctuations. For example, if one sets $Y_n:=\overline X_n+b/\sqrt{n}$, then the normalized variable $\sqrt n(Y_n-\mu)=\sqrt n(\overline X_n-\mu)+b$ converges in distribution to $\mathcal N(b,\sigma^2)$ as $n\to \infty$. The fluctuation has not changed; what has changed is the center. The normalized quantity retains a nonzero offset $b$ in the limit. 

This raises a similar question for arithmetic sums. Do sums of multiplicative functions exhibit similar centered or shifted behavior? This question is closely connected to recent problems in probabilistic and analytic number theory. 

In probabilistic number theory, partial sums of random multiplicative functions are a central and difficult topic, and they have important connections with harmonic analysis \cite{harper2015note}. Recent breakthroughs in this area concern Helson's conjecture and the moment theory of random multiplicative functions \cite{HARPER_2020, harper2015note, harper2010limit,harper2020moments,harper2023almost}.

There is also an interesting deterministic side to the question. Sums of the classical M\"obius, Liouville, and generalized Liouville functions have been studied extensively due to their connections with the Riemann hypothesis \cite{ingham1942two,humphries2013distribution,mossinghoff2017liouville} and the distribution of the number of prime divisors in residue classes \cite{coons2011residue, humphries2019biases}.

In this paper, we develop a bias theory for a broad class of multiplicative functions, the \emph{unimodular fake M\"obius functions}, whose values at a natural number $n$ depend only on the exponents in the prime factorization of $n$, and not on the prime factors themselves. This class contains many interesting examples, including the generalized Liouville functions.

\subsection{Bias phenomena in arithmetic functions} The probabilistic notion of asymptotic bias has a direct analogue in arithmetic, where one studies arithmetic partial sums and their sign changes. A classical example is the Mertens partial sum
\begin{equation}\label{mertens}
    M(x):=\sum_{n\le \lfloor x \rfloor}\mu(n),
\end{equation}
where $\mu$ is the M\"obius function, defined by $\mu(n)=0$ if $n$ is divisible by the square of a prime, and $\mu(n)=(-1)^k$ if $n$ is a product of $k$ distinct primes. The size and sign changes of $M(x)$ have a long history, going back to the work of Mertens, Landau, and Ingham \cite{mertens1897zahlentheoretische,ingham1942two}. Heuristically, one expects $M(x)/\sqrt{x}$ to oscillate around $0$ without any preference for either sign, in contrast with the Liouville sum discussed below. Mertens computed its values extensively and conjectured the bound $|M(x)|\le \sqrt{x}$ \cite{mertens1897zahlentheoretische}, a conjecture now known to be false \cite{odlyzko1984disproof}. What made this conjecture so compelling is that a bound of this shape would imply the Riemann hypothesis (RH) and the simplicity of the zeros of $\zeta(s)$ (SZC) \cite{ingham1942two}. Mertens's conjecture was eventually disproved by Odlyzko and te Riele \cite{odlyzko1984disproof}.

A closely related example is the Liouville summatory function 
\begin{equation}
    L(x):=\sum_{n\le \lfloor x \rfloor}\lambda(n),
\end{equation}
defined as the partial sum of the classical Liouville function $\lambda(n):=(-1)^{\Omega(n)},$ where $\Omega(n)$ is the total number of prime factors of $n$ counted with multiplicity \cite{humphries2013distribution}. Historically, P\'olya observed that $L(x)$ takes negative values for a large range of $x$ and asked whether $L(x)$ is non-positive for all $x\ge2$. He also noted that an affirmative answer would imply the Riemann hypothesis \cite{polya1919verschiedene}. Haselgrove later proved the answer is negative and showed that $L(x)$ changes sign infinitely often \cite{Haselgrove}. Thus $L(x)$, like $M(x)$, oscillates, but it appears to have a negative bias (see below).

The source of this contrast is already apparent at the level of Dirichlet series 
\begin{equation}
\sum_{n\ge1}\frac{\mu(n)}{n^s}=\frac{1}{\zeta(s)},
\qquad
\sum_{n\ge1}\frac{\lambda(n)}{n^s}=\frac{\zeta(2s)}{\zeta(s)}.
\end{equation}
Assuming RH and SZC, we have the explicit formulas
\begin{equation} \label{eq:M explicit formula}
\frac{M(x)}{\sqrt{x}}
= \sum_{\substack{\rho=\frac{1}{2}+i\gamma\\|\gamma|\le T}}
\frac{x^{i\gamma}}{\rho \zeta'(\rho)}+\mathcal E_M(x,T),
\end{equation}
and
\begin{equation}\label{eq:L explicit formula}
\frac{L(x)}{\sqrt{x}}
= \frac{1}{\zeta(\tfrac12)}
+\sum_{\substack{\rho=\frac{1}{2}+i\gamma\\|\gamma|\le T}}
\frac{\zeta(2\rho) x^{i\gamma}}{\rho \zeta'(\rho)}+\mathcal E_L(x,T),
\end{equation}
where the sums are over the nontrivial zeros $\rho$ of $\zeta(s),$ and $\mathcal E_M(x,T)$ and $\mathcal E_L(x,T)$ are error terms \cite{ng2004distribution, humphries2013distribution}. The key structural difference is the constant term $1/\zeta(\tfrac12)$ in the Liouville formula, coming from the pole of $\zeta(2s)/\zeta(s)$ at $s=\tfrac12$. No analogous term appears for $M(x)$, since $1/\zeta(s)$ has no pole there. Heuristically, this suggests that $M(x)/\sqrt{x}$ oscillates around $0$ (asymptotically unbiased), while $L(x)/\sqrt{x}$ oscillates around the value $\zeta(\tfrac12)^{-1}=-0.68476523 \dots$ (asymptotically biased).

\begin{figure}[!ht]
    \centering
    \includegraphics[width=0.5\linewidth]{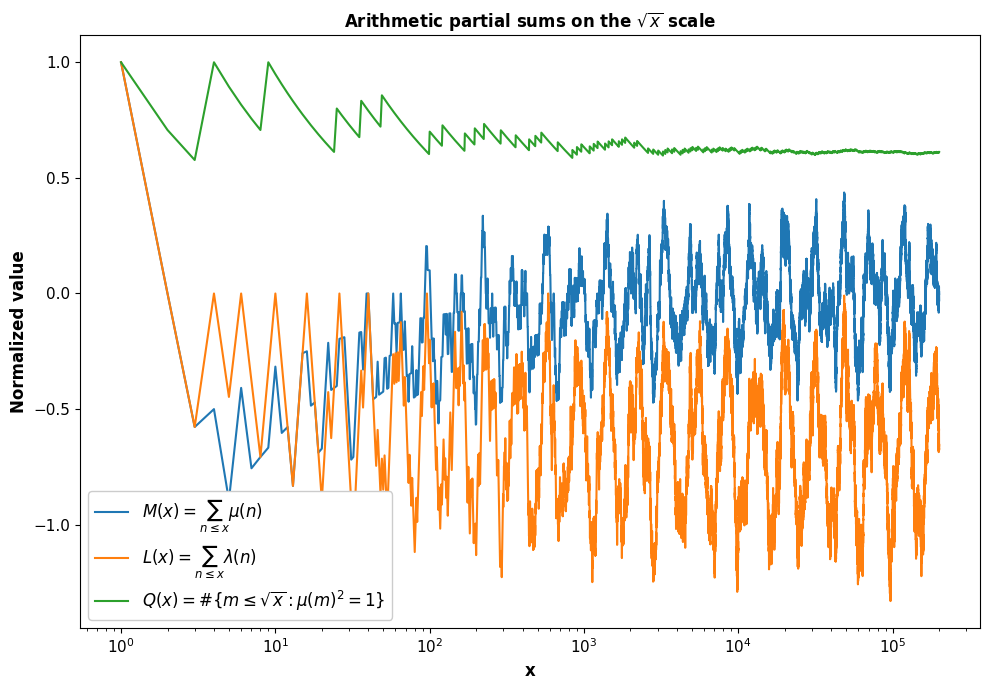}
    \caption{\small Plots of the normalized Mertens, Liouville, and squarefree counting summatory functions on the $\sqrt{x}$-scale, illustrating \emph{unbiased, apparent}, and \emph{persistent} behavior, respectively.}
    \label{fig:1}
\end{figure}

There is a third type of behavior that is not captured by the above two examples. To illustrate, consider the squarefree counting summatory function 
\begin{equation}
    Q(x):=\#\{m\le \sqrt{x}:\mu(m)^2=1\}=\sum_{m \le \sqrt x} \mu^2(m).
\end{equation}
It is well known that the squarefree numbers have density $6/\pi^2= 0.607927 \dots$ \cite{jakimczuk2013simple}. Therefore, unlike the Liouville function, where the normalized sum oscillates but does not converge, $Q(x)/ \sqrt{x}$ converges as $x \to \infty$ to the nonzero limit $6/\pi^2.$  

\Cref{fig:1} illustrates the three different behaviors at the scale $\sqrt{x}$: $M(x)$ appears \emph{unbiased} ($M(x)/\sqrt{x}$ oscillates around $0$), $L(x)$ appears \emph{apparently biased} ($L(x)/\sqrt{x}$ oscillates around the shifted level $\zeta(\tfrac12)^{-1}$), and $Q(x)$ appears \emph{persistently biased} ($Q(x)/\sqrt{x}$ converges to $6/\pi^2$). Formal definitions will be given in a later section.

\subsection{Unimodular fake M\"obius functions and zeta-factorization}

Recall that $\mu(n)$ is a multiplicative function that is completely determined by its prime-power values
\begin{equation}
    \mu(1)=1,\qquad \mu(p)=-1,\qquad \mu(p^k)=0 (k\ge2),
\end{equation}
and these values do not depend on the underlying prime. Equivalently, $\mu(n)$ is completely determined by the Euler product identity
\begin{equation}
\sum_{n\ge 1}\frac{\mu(n)}{n^s} =\frac{1}{\zeta(s)} =\prod_{p}\bigl(1-p^{-s}\bigr), \qquad \Re(s)>1,
\end{equation}
see for instance \cite[Sec. 2.2]{apostol2013introduction}.
This suggests a broad class of multiplicative functions obtained by prescribing their values on prime powers.
Following the ``fake $\mu$'' philosophy of Martin-Mossinghoff-Trudgian \cite{martin2023fake}
(and the subsequent oscillation theory in \cite{martin2025oscillation}),
we consider multiplicative functions $\mathfrak{f}$ whose prime-power values are prescribed by a fixed sequence
$(\varepsilon_k)_{k\ge0}$ with
\begin{equation}\label{eq: epsk seq}
\varepsilon_0=1,
\qquad \varepsilon_k\in\mathbb{S}^1\cup\{0\}\ (k\ge1),
\qquad \mathfrak{f}(p^k)=\varepsilon_k\ \forall\ k\ge 1, \forall\ p\ \text{prime}.
\end{equation}

This class contains the three examples discussed above:
\begin{itemize}
    \item The M\"obius function $\mu$ corresponds to $\varepsilon_1=-1$ and
    $\varepsilon_k=0$ for $k\ge2$.
    
    \item The Liouville function $\lambda(n)=(-1)^{\Omega(n)}$ corresponds to
    $\varepsilon_k=(-1)^k$ for all $k\ge0$.
    
    \item The fake mu corresponding to the squarefree counting example $Q(x)$ is the one determined by
    $\varepsilon_0=1$, $\varepsilon_1=0$, $\varepsilon_2=1$, and
    $\varepsilon_k=0$ for $k\ge3$.
\end{itemize}

For a given fake mu $\mathfrak{f}$, define the associated Dirichlet series\footnote{The second equality follows by the multiplicativity of $\mathfrak{f}$ and the absolute convergence of $F_{\mathfrak{f}}(s)$ on $\Re(s) >1,$ since $|\mathfrak{f}(n)| \le 1$ for all $n\ge 1.$}
\begin{equation}
F_{\mathfrak{f}}(s):= \sum_{n\ge 1}  \frac{\mathfrak{f} (n)}{n^s} =\prod_p \sum_{k\ge0}\varepsilon_k p^{-ks},
\qquad \Re(s)>1.
\end{equation}

A central fact used throughout the paper is that the above Euler product admits a factorization into complex powers of $\zeta(s).$

\begin{thm}
Let $\mathfrak{f}$ be a unimodular fake M\"obius function determined by \eqref{eq: epsk seq}.
Set
\begin{equation}\label{eq: z w param }
z:=\varepsilon_1,
\qquad
w:=\varepsilon_2-\frac{\varepsilon_1(\varepsilon_1+1)}{2}.
\end{equation}
Then
\begin{equation}\label{eq: Ff factor}
F_{\mathfrak{f}}(s)=\zeta(s)^{z} \zeta(2s)^{w} G_{\mathfrak{f}}(s),
\end{equation}
where $G_{\mathfrak{f}}(s)$ is a holomorphic Euler product on $\Re(s)>1/3$, and is uniformly bounded on $\Re(s)\ge \sigma_1$ for each fixed $\sigma_1>1/3$.
\end{thm}

\begin{remark}
Throughout the remainder of the paper, the symbols $z$ and $w$ will always denote the parameters defined in \eqref{eq: z w param }.
\end{remark}

\subsection{Main results}

\subsubsection{The analytic theory} Martin-Mossinghoff-Trudgian \cite{martin2023fake} initiated the study of
fake $\mu$'s with coefficients $\varepsilon_k\in\{-1,0,1\}$, focusing on the summatory function
\begin{equation}\label{eq:summatory}
A_{\mathfrak{f}}(x):=\sum_{n\le x}\mathfrak{f}(n).
\end{equation}
They introduced the notions of \emph{persistent bias} and \emph{apparent bias} for $A_{\mathfrak{f}}(x)$ at the scale $x^{1/2}$ based on an asymptotic formula, valid for large $x$, of the form
\begin{equation}
    A_{\mathfrak{f}}(x)= ax+bx^{1/2} + \mathcal{E}(x),
\end{equation}
where $a,  b \in \R$ and ``$\mathcal{E}(x)$ is an error term that either tends to $0$ in the limit,
or is expected to oscillate about $0$ in a roughly balanced manner'' \cite[Abstract]{martin2023fake}.  More recently, Martin-Yip \cite{martin2025oscillation} developed
a more general theory on the oscillatory behavior of the summatory function of fake $\mu$'s with $\varepsilon_k\in\{-1,0,1\}$ at a general scale $x^{1/(2\ell)}$ for $\ell \ge 1$.

Our work studies the substantially larger class \eqref{eq: epsk seq}, in which the parameters $\varepsilon_k$ are allowed to be arbitrary complex phases on $\mathbb{S}^1$. This includes natural examples such as the generalized Liouville functions $\mathfrak{f}(n)=\xi^{\Omega(n)}$, where $\xi$ is a root of unity. In this setting, complex powers of $\zeta(s)$ become an essential part of the analysis and introduce new analytic difficulties. Several other examples are discussed in Section \ref{sec:examples}.

Using the factorization \eqref{eq: Ff factor} and a suitable contour, we derive an explicit formula for the smoothed summatory function
\begin{equation}
    A_{\mathfrak{f}}^{\exp}(x)= \sum_{n \ge 1} \mathfrak{f}(n)   e^{-n/x}.
\end{equation}

Before we state the main analytic theorems, we recall that it is standard that there exist arbitrarily large $T$ such that 
\begin{equation}\label{eq: height T lower bd}
  |T-\gamma|  \gg  \frac{1}{\log T}
  \qquad\text{for every zero } \rho=\tfrac12+i\gamma,
\end{equation}
see, for example, \cite[p. 126]{gonek1984mean}.

\begin{thm}[Explicit formula for $\displaystyle A_{\mathfrak{f}}^{\exp}$]
\label{thm: exp form Af}
Assume RH and SZC. Assume in addition that $z,w \not \in \Z$, $-1<\Re z<1,$ and $\Re(w)<1$. Fix $\tfrac13<a<\tfrac12$ and $c>1$, and write $L:=\Log x$. Then there exists a constant $T_0=T_0(a,c,\mathfrak{f})>0$ such that for every $x \ge 3$ and every $T\ge T_0$ satisfying \eqref{eq: height T lower bd}, we have 
\begin{equation}\label{eq: exp form Af}
A_{\mathfrak{f}}^{\exp}(x)-\Delta_1(x) =\Delta_{1/2}(x) +\sum_{\substack{\rho=\frac{1}{2}+i\gamma\\|\gamma|<T}}\Delta_\rho(x)  +O\bigl(x^c e^{-T}+x^a\bigr).
\end{equation}
In particular, if $T>(c-a)\Log x$ and $x$ is large, then the error term in \eqref{eq: exp form Af} is $O(x^a).$

Moreover, each main term has a Laplace representation
\begin{equation} \label{eq: Delta xi}
\Delta_\xi(x)=\kappa_\xi x^\xi\int_0^{\beta_\xi} e^{-Lu} u^{-\alpha_\xi} \mathcal J_\xi(u;\mathfrak{f}) du,
\qquad
\xi\in\Bigl\{1,\tfrac12\Bigr\}\cup\{\rho:|\Im\rho|<T\},
\end{equation}
with $\mathcal J_\xi( \cdot ;\mathfrak{f})$ holomorphic near $u=0$ and
\begin{equation}
(\kappa_1,\alpha_1,\beta_1)=\Bigl(\tfrac{\sin(\pi z)}{\pi}, z, \tfrac12\Bigr),\qquad
(\kappa_{1/2},\alpha_{1/2},\beta_{1/2})=\Bigl(\tfrac{\sin(\pi(z+w))}{\pi} 2^{-w+1}, w, \tfrac12-a\Bigr),
\end{equation}
\begin{equation}
(\kappa_{\rho},\alpha_{\rho},\beta_{\rho})=\Bigl(-\tfrac{\sin(\pi z)}{\pi}, -z, \tfrac12-a\Bigr),\qquad \rho=\tfrac12+i\gamma.
\end{equation}
The integrands $\mathcal J_\xi$ are given explicitly in Section \ref{sec: the integrands}.
\end{thm}

\begin{remark}
When  $z$ or $w$ is an integer, some of the main terms $\Delta_\xi$ are given by a residue term 
\begin{equation}
    \Delta_{\xi}(x)= \operatorname{Res}_{s=\xi} \Bigl( F_{\mathfrak{f}}(s) \Gamma(s) x^s\Bigr).
\end{equation}
In \cite{martin2023fake}, Martin-Mossinghoff-Trudgian covered the cases where $z \in \{-1,0,1\}$ and $w \in \{-2,-1,0,1\}$.
\end{remark}

\begin{remark}
Based on the definition of $z,w$ in \eqref{eq: z w param }, it can be shown that they must satisfy $\Re(z)\in[-1,1]$ and $\Re(w) \in[-2, \frac{25}{16}]$ (see Section \ref{sec: parameter ranges}). The present paper covers the parameter region
\begin{equation}
    \{(z,w): z, w \not \in \Z, \quad -1< \Re(z) < 1, \quad \Re(w) <1\}.
\end{equation}
The complementary parameter region requires a different analytic method and will be discussed in a future paper.
\end{remark}

\begin{thm}[Asymptotic expansions for $\Delta_\xi(x)$]
\label{thm: watson exp for Delta}
Assume the setup of Theorem \ref{thm: exp form Af}.
For each fixed $\xi\in\{1,\tfrac12\}\cup\{\rho:|\gamma|<T\}$, write $\mathcal J_\xi(u;\mathfrak{f})=\sum_{k\ge0}\lambda_{\xi,k} u^k$. Then for every fixed $M \in \N$, we have, for all large $x,$
\begin{equation}
\Delta_\xi(x)=\kappa_\xi x^\xi\bigl(
\sum_{k=0}^{M}\lambda_{\xi,k} \Gamma(1-\alpha_\xi+k) (\Log x)^{-(1-\alpha_\xi+k)}
+O_{M} \bigl((\Log x)^{-(2-\Re\alpha_\xi+M)}\bigr)\bigr),
\end{equation}
with $(\kappa_\xi,\alpha_\xi)$ as in Theorem \ref{thm: exp form Af}. The coefficients $\lambda_{\xi,k}$ are obtained by expanding the functions $\mathcal J_\xi(u;\mathfrak{f})$ from Theorem \ref{thm: exp form Af} at $u=0$.
\end{thm}

\subsubsection{The bias theory} As an application of the analytic theory, we introduce the notions of \emph{persistent bias} and \emph{apparent bias} at the scale
\begin{equation}\label{eq:scale}
x^{1/2}(\Log x)^{w-1}, \qquad w=\varepsilon_2 - \frac{\varepsilon_1(\varepsilon_1+1)}{2}
\end{equation}
which is precisely the order of the main term of $\Delta_{1/2}(x)$ in the explicit formula \eqref{eq: exp form Af} (see \Cref{thm: watson exp for Delta}). In particular, we formalize the definition of \emph{apparent bias} and the notion of oscillation in a ``roughly balanced manner'' introduced in \cite{martin2023fake}. Similar definitions also appear in the studies of related bias phenomena \cite{rubinstein1994chebyshev}.

\begin{defn}[Bias at the scale $x^{1/2}(\Log x)^{w-1}$] \label{def: bias }
Define the normalized quantity
\begin{equation}\label{eq: Bf def}
B_{\mathfrak{f}}^{\exp}(x;z,w)
:=\frac{A_{\mathfrak{f}}^{\exp}(x)-\Delta_1(x)}{x^{1/2}(\Log x)^{w-1}},
\qquad x\ge3,
\end{equation}
and its logarithmic mean
\begin{equation}
\langle B_{\mathfrak{f}}^{\exp}\rangle_{\log}(X)
:=\frac{1}{\Log X}\int_2^X \frac{B_{\mathfrak{f}}^{\exp}(t;z,w)}{t} dt.
\end{equation}
Let $b\in\C$ with $b\neq0$.
\begin{itemize}
\item We say $A_{\mathfrak{f}}^{\exp}$ has \emph{persistent (limiting) bias} $b$ if $\displaystyle \lim_{x\to\infty}B_{\mathfrak{f}}^{\exp}(x;z,w)=b$.
\item We say $A_{\mathfrak{f}}^{\exp}$ has \emph{apparent bias} $b$ if $B_{\mathfrak{f}}^{\exp}(x;z,w)$ is bounded, the pointwise limit of $B_{\mathfrak{f}}^{\exp}(x;z,w)$ as $x\to \infty$ does \emph{not} exist, and 
\begin{equation}
    \langle B_{\mathfrak{f}}^{\exp}\rangle_{\log}(X)\to b, \qquad X\to\infty.
\end{equation}
\item We say $A_{\mathfrak{f}}^{\exp}$ has \emph{no bias} at this scale if it has no nonzero persistent bias and no nonzero apparent bias.
\end{itemize}
\end{defn}

The bias constant at this scale comes from the leading Watson coefficient of $\Delta_{1/2}(x),$ appearing in \Cref{thm: watson exp for Delta}.
Write
\begin{equation}\label{eq: c12 def}
c_{1/2}(\mathfrak{f};z,w)
:= \frac{\sin\bigl(\pi(z+w)\bigr)}{\pi} 2^{-w+1} 
\Gamma(1-w) \lambda_{1/2,0}(\mathfrak{f};z,w),
\end{equation}
where $\lambda_{1/2,0}(\mathfrak{f};z,w):=\lambda_{1/2,0}$ is defined in \Cref{thm: watson exp for Delta}. In Section \ref{sec: cst c12}, we give an explicit formula for $c_{1/2}(\mathfrak{f};z,w)$.

To establish the bias results, we assume, in addition, the weak Mertens hypothesis (WMH), which asserts that 
\begin{equation}\label{eq:WMH1}
    \int_1^X\bigl(\frac{M(x)}{x}\bigr)^2 dx \ll \Log X,
\end{equation} 
where $M$ is the Mertens function defined in \eqref{mertens}. The main advantage of this hypothesis is that it implies a lower bound on the distance between consecutive zeros of $\zeta(s).$ In fact, we do not need the full strength of this hypothesis to prove the result (see Section \ref{sec: WMH}), but we shall assume it here for the sake of simplicity.

\begin{thm}[Bias at the scale $x^{1/2} (\Log x)^{w-1}$]\label{thm: bias theorem}
Assume RH, SZC, and the WMH \eqref{eq:WMH1}.
Assume also that $z,w\not \in\Z$, $-1<\Re z<1$, and $\Re(w)<1$. Then one has an expansion of the form
\begin{equation}\label{eq: Bf expan}
B_{\mathfrak{f}}^{\exp}(x;z,w)
=
c_{1/2}(\mathfrak{f};z,w)
+
(\Log x)^{-(z+w)}
\sum_{\rho}\frac{\Delta_\rho(x)}{x^{1/2}(\Log x)^{-(z+1)}}
+o(1), \qquad x \to \infty.
\end{equation}
Moreover, the sum over the zeros $\rho$ of $\zeta$ in \eqref{eq: Bf expan} is not identically
zero, and the following hold at the scale \eqref{eq:scale}.
\begin{enumerate}
\item If $\Re(z+w)>0$ and $c_{1/2}(\mathfrak{f};z,w)\neq0$, then $A_{\mathfrak{f}}^{\exp}$ has \emph{persistent bias} $c_{1/2}(\mathfrak{f};z,w)$.
\item If $\Re(z+w)=0$ and $c_{1/2}(\mathfrak{f};z,w)\neq0$, then $A_{\mathfrak{f}}^{\exp}$ has \emph{apparent bias} $c_{1/2}(\mathfrak{f};z,w)$.
\item In all remaining cases, there is no nonzero bias at this scale. In particular, if $c_{1/2}(\mathfrak{f};z,w)=0$ there is no nonzero persistent or apparent bias. And, if $\Re(z+w)<0$, then $B_{\mathfrak{f}}^{\exp}(x;z,w)$ is unbounded as $x\to\infty$.
\end{enumerate}
\end{thm}

In Section \ref{sec: cst c12}, we give an explicit criterion for $c_{1/2}(\mathfrak{f};z,w)=0.$

\subsection{Examples}\label{sec:examples}
We now illustrate the range of the class of multiplicative functions covered by our analytic and bias theory by describing several useful families of fake $\mu$'s. We also give some numerical illustrations of the bias phenomena. 
In each figure, we plot the complex-plane trajectory of the centered summatory function
$
B_{\mathfrak{f}}^{\exp}(x;z,w)-c_{1/2}(\mathfrak{f};z,w),
$
and we mark the origin by $\times$.
\begin{enumerate}
\item \emph{Completely multiplicative.}
Fix $\xi\in\mathbb{S}^1$ and take $\varepsilon_k=\xi^k$ for $k\ge0$.  Then
$\mathfrak{f}$ is completely multiplicative with $\mathfrak{f}(n)=\xi^{\Omega(n)}$ and the Dirichlet series is
\begin{equation}
F_{\xi}(s):=F_{\mathfrak{f}}(s) =\sum_{n\ge1}\frac{\xi^{\Omega(n)}}{n^s} =\prod_p\frac{1}{1-\xi p^{-s}}, \qquad \Re(s)>1.
\end{equation}

\smallskip
This family contains several interesting cases.

\smallskip
\begin{itemize}
\item If $\xi=1$, we have $\mathfrak{f}(n)\equiv 1$ and $F_{\xi}(s)=\zeta(s)$.

\item If $\xi=-1$, we have $\mathfrak{f}(n)=(-1)^{\Omega(n)}=\lambda(n)$ (Liouville).

\item If $\xi$ is an $m$th root of unity, we obtain the generalized Liouville functions. It is easy to observe that $\xi^{\Omega(n)}$ depends only on $\Omega(n)\ (\mathrm{mod}\ m)$. By Fourier inversion modulo $m,$ the partial sums of these weights encode the distribution of $\Omega(n)$ among residue classes modulo $m$ (for details, see, e.g., \cite{coons2011residue}). Using the cyclotomic factorization
$\prod_{j=0}^{m-1}(1-\xi^j x)=1-x^m$, we have
\begin{equation}
\prod_{j=0}^{m-1} F_{\xi^j}(s) =\prod_p\prod_{j=0}^{m-1}\frac{1}{1-\xi^j p^{-s}} =\prod_p\frac{1}{1-p^{-ms}} =\zeta(ms).
\end{equation}
\end{itemize}
\begin{figure}[!ht]
    \centering
    \includegraphics[width=0.9\linewidth]{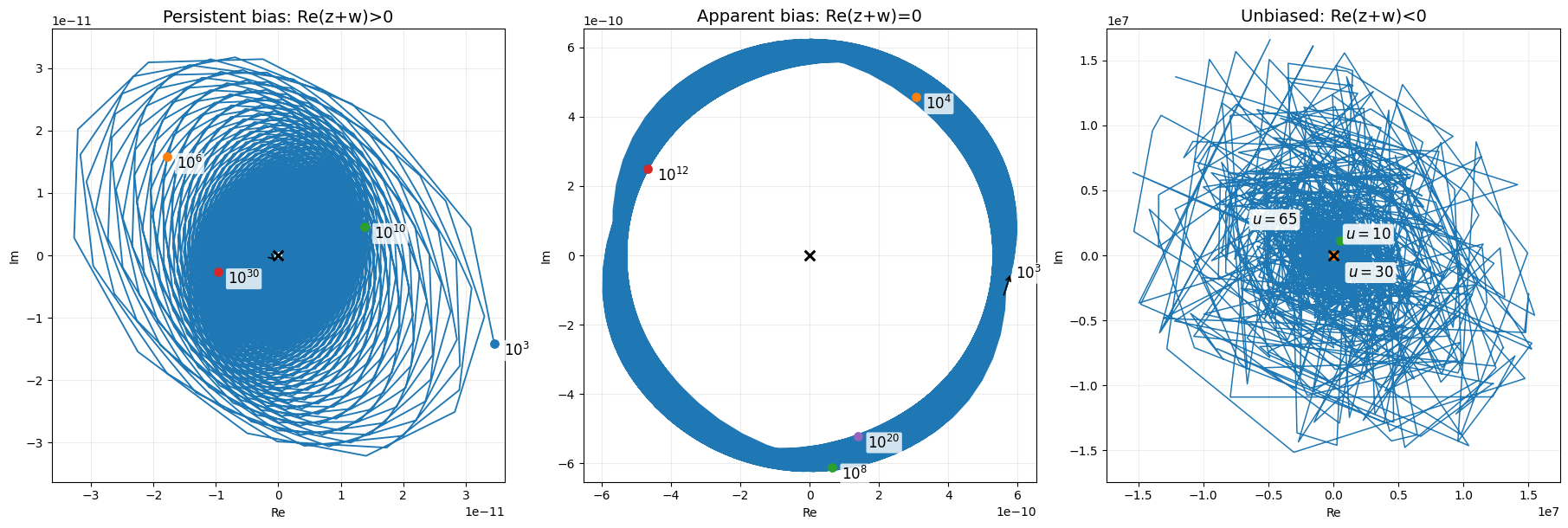}
    \caption{(Completely multiplicative $\mathfrak{f}):$ Take $\varepsilon_k=e^{ik\theta}$ for all $k\ge0$. Shown are the trajectories corresponding to the three angles $\theta\in\{\pi/5,\ \pi/3,\ 2\pi/3\}$, with $\Re(z+w)=0.5590,\ 0,\ -0.5$, illustrating nonzero persistent bias, apparent bias, and no bias, respectively.}
\end{figure}

\item \emph{Periodic $\varepsilon_k$.}
If $\varepsilon_{k+m}=\varepsilon_k$ for $k\ge1$, then
\begin{equation}
g_{\mathfrak{f}}(u)=1+\frac{\varepsilon_1u+\cdots+\varepsilon_mu^m}{1-u^m},
\end{equation}
so each Euler factor is rational in $p^{-s}$. Consider the following examples. 
\begin{itemize}
\item Let $m\ge2$ and set
\begin{equation}
    \varepsilon_k=
    \begin{cases}
    1,& m\mid k,\\
    0,& m\nmid k.
    \end{cases}
\end{equation}
This is periodic with period $m$ since $m \mid k$ iff $m \mid (k+m).$ Here, $\mathfrak{f}(n)=1$ if $n$ is a perfect $m$th power, and $\mathfrak{f}(n)=0$ otherwise. Thus, in this case,
\begin{equation}
F_{\mathfrak{f}}(s)=\prod_p\frac{1}{1-p^{-ms}}=\zeta(ms).
\end{equation}
\item Let $\chi$ be a Dirichlet character modulo $m$, and set
\begin{equation}
    \varepsilon_0=1,\qquad \varepsilon_k=\chi(k),\qquad k\ge1.
\end{equation}
Then we have an Euler product whose prime powers are twisted by $\chi$
\begin{equation}
F_{\mathfrak{f}}(s)= \prod_p \bigl(1+\sum_{k\ge1}\chi(k)p^{-ks} \bigr).
\end{equation}
\item Take the sequence with period $2$, i.e. $\varepsilon_{k+2}=\varepsilon_k,$ with $(\varepsilon_1,\varepsilon_2):=(i,-i).$ Equivalently, for every prime $p$ and every $k\ge1$, define
\begin{equation}
\mathfrak{f}(p^k):=
\begin{cases}
i,& k\ \text{odd},\\
-i,& k\ \text{even}.
\end{cases}
\end{equation}
Then 
\begin{equation}
    F_{\mathfrak{f}}(s)=\zeta(2s)\prod_p\Bigl(1+i p^{-s}-(1+i)p^{-2s}\Bigr).
\end{equation}

\begin{figure}[!ht]
    \centering
    \includegraphics[width=0.35\linewidth]{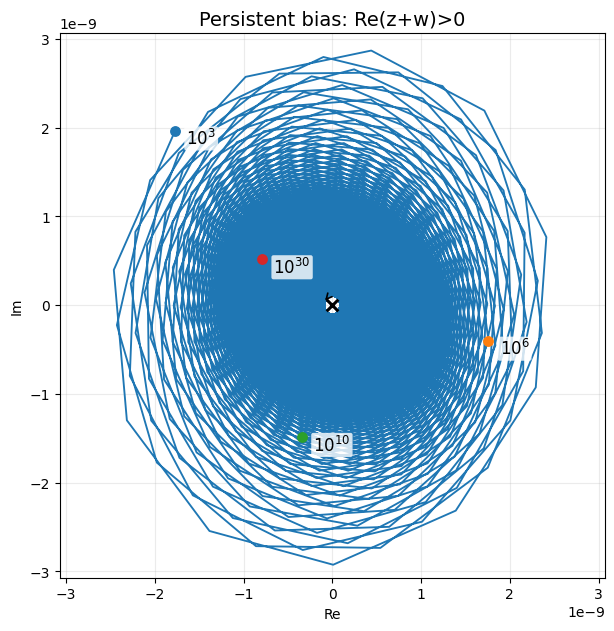}
    \caption{(Periodic $\varepsilon_k):$ Take the $2$-periodic sequence $\varepsilon_{2\ell+1}=i$ and $\varepsilon_{2\ell+2}=-i$, so that $\Re(z+w)=\tfrac12>0$. Hence $A_{\mathfrak{f}}^{\exp}$ has persistent bias \protect\footnotemark.}
\end{figure}
\end{itemize}
\footnotetext{Other behaviors are also possible in this family of fake $\mu$'s.}

\item \emph{Finitely supported $\varepsilon_k$.}
If $\varepsilon_k=0$ for all $k>m$, then $g_\mathfrak{f}$ is a polynomial of degree $m$ and
\begin{equation}
F_{\mathfrak{f}}(s)=\prod_{p}\Bigl(1+\varepsilon_1 p^{-s}+\cdots+\varepsilon_m p^{-ms}\Bigr).
\end{equation}
Interesting examples include the following.
\begin{itemize}
    \item If $m=1$, $\varepsilon_1=-1$ gives $F_{\mathfrak{f}}(s)=1/\zeta(s)$ and $\mathfrak{f}$ is the classical M\"obius function.
    \item Fix $r\ge2$ and $\xi\in\mathbb{S}^1$, and take
    \begin{equation}
        \varepsilon_k=
        \begin{cases}
        \xi^k, & 0\le k\le r-1,\\
        0, & k\ge r.
        \end{cases}
    \end{equation}
    Then
    \begin{equation}
        \mathfrak{f}(n)=
        \begin{cases}
        \xi^{\Omega(n)},& n\ \text{$r$-free},\\
        0,& \text{otherwise},
        \end{cases}
    \end{equation}
    where $n$ is called $r$-free if no prime $p$ satisfies $p^r \mid n.$ This is Tanaka's generalized M\"obius function $\mu_r$ \cite{tanaka1980mobius}. In this case,
    \begin{equation}
    F_{\mathfrak{f}}(s)= \prod_p \bigl( 1+\xi p^{-s}+\xi^2p^{-2s}+\cdots+\xi^{r-1}p^{-(r-1)s}\bigr)= \prod_p \frac{1-\xi^r p^{-rs}}{1-\xi p^{-s}}.
    \end{equation}
\begin{figure}[!ht]
    \centering
    \includegraphics[width=0.9\linewidth]{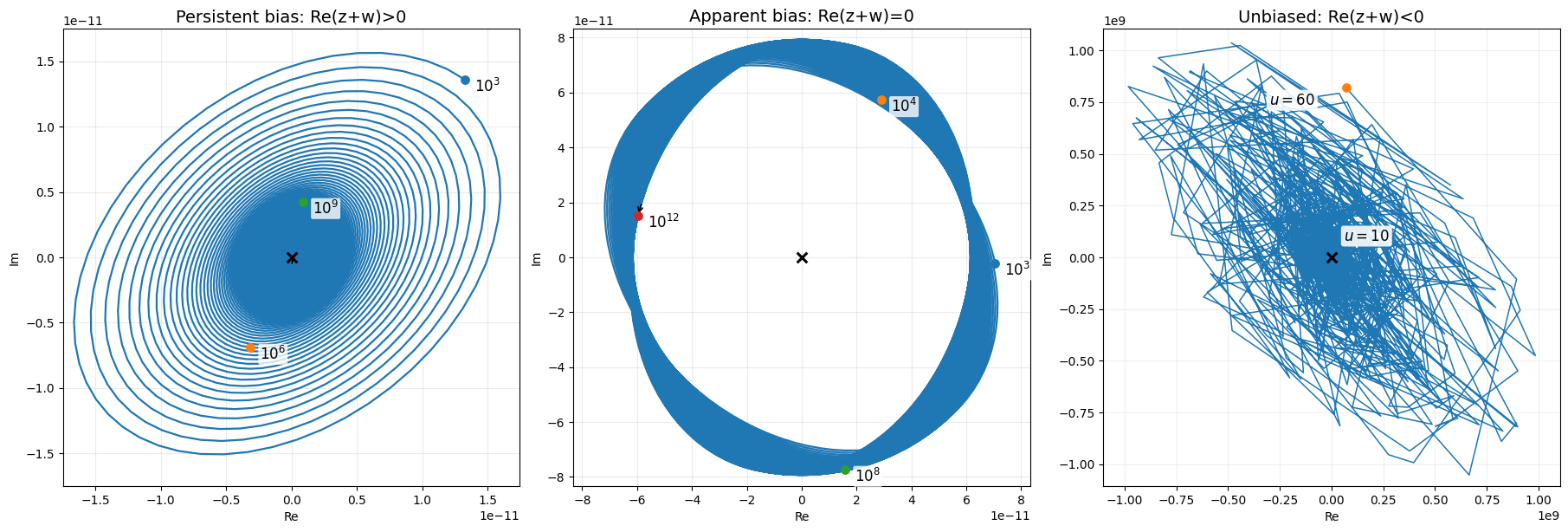}
    \caption{(Finitely supported $\varepsilon_k):$ Take
    $\varepsilon_1=e^{i\pi/5}$, $\varepsilon_2=\xi_2$, and $\varepsilon_k=0$ for $k\ge3$.
    Shown are the trajectories corresponding to $\xi_2\in\{1, -\tfrac14+i\sqrt{15}/4, -1\}$, with $\Re(z+w)=1.25,\ 0,\ -0.75$,
    illustrating nonzero persistent bias, apparent bias, and no bias, respectively.}
\end{figure}

\end{itemize}

\end{enumerate}

\subsection{Parameter ranges} \label{sec: parameter ranges}

The parameters $(z,w)$ in \eqref{eq: z w param } are determined by the first two coefficients $(\varepsilon_1,\varepsilon_2)$.
Since $\varepsilon_1\in\mathbb{S}^1\cup\{0\}$, we have $|z|= |\varepsilon_1|\le 1$ and therefore 
\begin{equation}\label{eq:rez}
    \Re(z)\in[-1,1].
\end{equation}
Moreover,
\begin{equation}
w=\varepsilon_2-\frac{\varepsilon_1(\varepsilon_1+1)}{2},
\end{equation}
so since $|\varepsilon_2|\le1,$ we have
\begin{equation}
    \Re(w)\le 1-\frac{1}{2} \Re(\varepsilon_1^2+\varepsilon_1).
\end{equation}
Thus, assuming $\varepsilon_1 \neq0,$ writing $\varepsilon_1=e^{i\theta},$ and setting $c:=\cos\theta\in[-1,1]$, we get
\begin{equation}
    \Re(\varepsilon_1^2+\varepsilon_1)=\cos(2\theta)+\cos\theta=(2c^2-1)+c,
\end{equation}
so that
\begin{equation}
    \Re(w)\le 1-\frac{2c^2+c-1}{2}
=\frac32-c^2-\frac{c}{2}
=\frac{25}{16}-\Bigl(c+\frac{1}{4}\Bigr)^2
\le \frac{25}{16}.
\end{equation}
The upper bound is achieved when $\varepsilon_2=1$ and
 $\cos\theta=-\frac{1}{4}$.

On the other hand, 
\begin{equation}
    |w|=|\varepsilon_2-\frac{\varepsilon_1(\varepsilon_1+1)}{2}| \le 2,
\end{equation}
so $\Re(w) \ge -2,$ which is attained when $\varepsilon_1=1$ and $\varepsilon_2=-1.$

Therefore, 
\begin{equation}\label{eq:rew}
    \Re(w) \in[-2, \frac{25}{16}].
\end{equation}

In this paper, we focus on the subset
\begin{equation}
z,w \not \in \Z, \qquad -1< \Re(z)< 1,\qquad \Re(w) < 1,
\end{equation}
of the parameter set given by \eqref{eq:rez} and \eqref{eq:rew}. The remaining cases require a different analytic method and will be discussed in a future paper.

\subsection{Organization of the paper}

\begin{itemize}
    \item Section \ref{sec: zeta z} presents estimates, branch conventions, and local factorizations of complex powers of $\zeta(s)$.
    \item Section \ref{sec: fake mus} presents the definition of unimodular fake M\"obius functions, the Euler product factorization
\eqref{eq: z w param }, and a uniform bound for $G_{\mathfrak{f}}(s)$ in a right-half plane.
    \item Section \ref{sec: perron form} presents Perron's formula, the contour of integration, and the explicit formula of the summatory function $A_{\mathfrak{f}}^{\exp}$. 
    \item Section \ref{sec: proofs} presents the proofs of the main theorems.
\end{itemize}

\section{Complex powers of $\zeta(s)$} \label{sec: zeta z}

Throughout this section, write $s=\sigma+it$ and let $\Log$ denote the principal branch on $\C\setminus(-\infty,0]$. We assume the Riemann hypothesis (RH) and the simplicity of the zeros of $\zeta(s)$ (SZC), and write $\rho=\tfrac12+i\gamma$ for a nontrivial zero.

Fix $z \in \C.$ In this section, we study the analytic continuation of $\zeta(s)^z$ and its local factorizations near its singularities. We then establish some standard uniform bounds for $\zeta(s)^z$ away from its singularities.

\subsection{Definition and analytic continuation} \label{sec: zeta z analy cont}

For $\sigma>1$, define
\begin{align}
  \zeta(s)^{z} &:= \prod_{p}\bigl(1-p^{-s}\bigr)^{-z} \label{eq:zetaz rhp} \\
  &= \prod_{p}\left\{1+\sum_{\nu\ge1} \binom{z+\nu-1}{\nu} p^{-\nu s} \right\}, \qquad \sigma>1,
\end{align}
where the infinite product converges absolutely and locally uniformly on $\sigma>1$. See, for example, \cite[Sec. 5.1]{tenenbaum2015introduction}.

Let
\begin{equation} \label{eq: omega'}
    \Omega':= \bigl\{ \Re(s)>1/3 \bigr\} \setminus \bigcup_{\rho} (-\infty+i \Im\rho,\ \rho],
\end{equation}
and
\begin{equation}
  \Omega :=  \bigl\{ \Re(s)>1/3 \bigr\} \setminus \Bigl( (-\infty,1] \cup \bigcup_{\rho}(-\infty+i \Im\rho,\ \rho]\Bigr),
\end{equation}
where the unions are taken over the nontrivial zeros $\rho$ of $\zeta(s)$ and
\begin{equation}
  (-\infty+i \Im\rho,\ \rho]
  :=\{ t+i \Im\rho : -\infty<t\le\tfrac12 \}.
\end{equation}
Then $\zeta(s)$ is holomorphic and nonvanishing on $\Omega,$ and hence it admits a single-valued holomorphic logarithm there. The lower bound $\Re(s)>1/3$ is not essential here, but it will be convenient later in the paper.

We fix the branch by the normalization $\log\zeta(2)=\ln\zeta(2)$ and define, for $s\in\Omega$,
\begin{equation} \label{eq: log zeta and zeta z}
\log\zeta(s):=\ln\zeta(2)+\int_{2}^{s}\frac{\zeta'(\alpha)}{\zeta(\alpha)} d\alpha,
\qquad
\zeta(s)^z:=\exp \bigl(z \log\zeta(s)\bigr),
\end{equation}
where the integral is taken along any path in $\Omega$ from $2$ to $s$. If $z\not \in\Z$, then $\zeta(s)^z$ has branch points at the pole $s=1$ and at zeros $\rho$ of $\zeta(s)$.

\begin{remark}
It is easy to show that the two definitions \eqref{eq:zetaz rhp} and \eqref{eq: log zeta and zeta z} of $\zeta(s)^z$ agree on $\sigma>1$.
\end{remark}

Similarly, the function $(s-1)\zeta(s)$ is holomorphic and nonvanishing on $\Omega'$, so
we can choose a single-valued holomorphic logarithm of it.

\begin{defn}\label{def:LZ}
On $\Omega'$ we define $\mathrm{L}_1(s)$ to be the unique holomorphic logarithm of $(s-1)\zeta(s)$ with the normalization $\mathrm{L}_1(2):= \ln \zeta(2).$ That is,
\begin{equation}
        \mathrm{L}_1(s):= \ln \zeta(2)  + \int_2^s  \frac{ \bigl( (\alpha-1) \zeta(\alpha) \bigr)'}{ (\alpha-1)\zeta(\alpha)}  d \alpha,
\end{equation}
where the integral is taken over any path from $2$ to $s$ in $\Omega'.$

We also define
\begin{equation}
  Z(s;z):= \frac{\exp\bigl(z \mathrm{L}_1(s)\bigr)}{s}, \qquad s\in\Omega'.
\end{equation}
\end{defn}

\begin{remark}
    Since $\mathrm{L}_1(s)= \ln \bigl( (s-1)\zeta(s)\bigr)$ for $s \in (1,2]$ and $(s-1)\zeta(s) \to 1$ as $s \to1^+,$ we have $\mathrm{L}_1(1)=0$.
\end{remark}

\begin{defn}
    For $\alpha \in \C,$ we define 
\begin{equation}
  (s-1)^{\alpha} := \exp\bigl(\alpha \Log(s-1)\bigr),
  \qquad s\in \C \setminus (-\infty, 1],
\end{equation}
where $\Log$ denotes the principal branch on $\C\setminus(-\infty,0]$.
\end{defn}

\begin{prop}[Analytic continuation] \label{zetasdef}
For $s\in\Omega$,
\begin{equation}\label{eq: zetazdef}
\zeta(s)^z  =  s (s-1)^{-z} Z(s;z)= (s-1)^{-z}  \exp\Bigl(z  \mathrm{L}_1(s)\Bigr).
\end{equation}
We refer to this as the factorization of $\zeta(s)^z$ into its singular and holomorphic parts near $s=1.$

\emph{Proof.}  The identity holds on the line $\{s>1\}$ and the right hand side is holomorphic on $\Omega$, so it extends by the identity theorem. \qed
\end{prop}

\begin{remark}
    The factorization in Proposition \ref{zetasdef} also follows from
    \begin{equation} \label{eq: L1 def.}
        \mathrm{L}_1(s)= \ln \zeta(2)  + \int_2^s  \frac{ \bigl( (u-1) \zeta(u) \bigr)'}{ (u-1)\zeta(u)}  d u= \Log(s-1) + \log \zeta(s).
    \end{equation}
\end{remark}

\begin{remark}[Complex powers of $\zeta(2s)$]\label{rem:zeta2s}
Let
\begin{equation} \label{eq: omega2}
  \Omega_2 := \{ s\in\C : 2s\in\Omega \},\qquad
  \Omega_2' := \{ s\in\C : 2s\in\Omega' \}.
\end{equation}
For $w\in\C$ and $s\in\Omega_2$ we similarly define
\begin{equation}
  \zeta(2s)^{w}
  := 2s (2s-1)^{-w} Z(2s;w),
\end{equation}
and $Z(2s;w)$ is holomorphic on $\Omega_2'.$
\end{remark}

\subsection{Local factorization of $\zeta(s)^z$ near $\rho$} 
Fix a zero $\rho=\tfrac12+i \gamma$ of $\zeta(s)$. Since $\zeta(s)^z$ has a branch point at $\rho$ when $z\not \in\Z$, it is convenient to factor it locally near $\rho$ into a singular and holomorphic factor. 

Fix a disk $D_\rho$, centered at $\rho$, chosen small enough so that it does not contain any other zero of $\zeta(s).$ Let
\begin{equation}
  \Gamma_\rho := \{\rho-t:\ t\ge0\}.
\end{equation}  

Fix $s_\rho \in D_\rho \setminus \Gamma_\rho$. On $D_\rho,$ we define
\begin{equation}
    L_\rho(s ):= \mathrm{L}_1(s_\rho) -  \Log(s_\rho-\rho) + \int_{s_\rho}^{s}
       \bigl(
         \frac{1}{u-1} + \frac{\zeta'(u)}{\zeta(u)} - \frac{1}{u-\rho}
       \bigr) du.
\end{equation}

We note that $L_\rho(s )$ is a holomorphic logarithm of the function $\frac{(s-1)  \zeta(s)}{s-\rho}$ on $D_\rho.$ Moreover, it is easy to see from the definition of $L_1'(s)$ in \eqref{eq: L1 def.} that $L_\rho$ has the property

\begin{equation}\label{eq: Lrho}
  L_\rho(s) = \mathrm{L}_1(s) -  \Log(s-\rho), \qquad s\in D_\rho\setminus\Gamma_\rho.
\end{equation}

\begin{corollary}[Local factorization of $\zeta(s)^{z}$ near $\rho$]
\label{cor:localform}
For $s\in D_\rho \setminus\Gamma_\rho$ we have
\begin{equation} \label{eq: zeta z near rho}
  \zeta(s)^{z}
  = (s-1)^{-z} (s-\rho)^{ z} Z_\rho(s;z),
\end{equation}
where
\begin{equation}
  Z_\rho(s;z) := \exp\bigl(z L_\rho(s)\bigr)
\end{equation}
is holomorphic and never zero on $D_\rho$.
\end{corollary}

\begin{proof}
The claim follows from \eqref{eq: zetazdef} and \eqref{eq: Lrho}.
\end{proof}

\subsection{Bounds for $\zeta(s)^z$}

In this section, we present some standard bounds for $\zeta(s)^z$ that will be needed later. For the purposes of our contour estimates, subexponential bounds would already suffice. Nevertheless, we state the sharper available polynomial and subexponential bounds in the relevant regions.

We first recall that it is standard that there exist arbitrarily large $T>15$ such that 
\begin{equation}\label{eq: T lowe bd}
  |T-\gamma|  \gg  \frac{1}{\log T} \qquad\text{for every zero } \rho=\tfrac12+i\gamma,
\end{equation}
see, for example, \cite[p. 126]{gonek1984mean}. 

\begin{lemma}\label{lem:zeta z bounds}
Assume RH and SZC. Let $A>0$.
\begin{enumerate}
\item If $\sigma_0>1/2$, then for every $\varepsilon>0$,
\begin{equation}
\zeta(\sigma+it)^z\ll_{\sigma_0,A}(1+|t|)^\varepsilon
\end{equation}
uniformly for $\sigma\ge\sigma_0$, $|t|\ge3$, and $|z|\le A$.

\item Assume $T>15$ is a real number satisfying \eqref{eq: T lowe bd}. Then there exists a constant $C>0$ depending only on $A$ such that
\begin{equation}
\zeta(\sigma+iT)^z\ll T^{C\Log\Log T},
\end{equation}
uniformly for $-1\le\sigma\le2$ and $|z|\le A$.

\item If $0<a<1/2$, then on the punctured line
\begin{equation}
V_a=\{a+it: |t|\ge2\}\setminus\{a+i\gamma:\zeta(\tfrac12+i\gamma)=0\},
\end{equation}
we have
\begin{equation}
\zeta(s)^z\ll_{a,A}|t|^{C_{a,A}}.
\end{equation}
\end{enumerate}
\end{lemma}

\begin{proof}
The first estimate follows from the uniform bound
$\log\zeta(\sigma+it)=o(\log |t|)$ as $|t| \to \infty,$ valid in any fixed half-planes
$\sigma\ge\sigma_0>1/2$ (see \cite[Notes to Ch. XIV, Eq. (14.33)]{titchmarsh1986theory} and \cite[Eq. (5.10)  p. 280]{tenenbaum2015introduction}). The second and third follow from the standard estimate \cite[Thm. 9.6(B)]{titchmarsh1986theory} 
\begin{equation}
\log\zeta(s)=\sum_{|t-\gamma|\le1}\Log(s-\rho)+O(\log |t|),
\end{equation}
valid uniformly for $-1\le\sigma\le2$ when $t$ is not a zero ordinate. Indeed, for $|t-\gamma|\le1$ we use
\begin{equation}
    (\Log T)^{-1} \ll  |s-\rho| \le 3, \qquad 1/2 -a \le |s-\rho|\le  \sqrt{(1/2 -a)^2 +1},
\end{equation}
for the second and third estimates, respectively. The claims then follow by the Riemann-von Mangoldt formula \cite[Ch. 9]{titchmarsh1986theory}. 
\end{proof}

\section{Unimodular fake M\"obius functions} \label{sec: fake mus}
\subsection{Definition}
Let $(\varepsilon_k)_{k\ge1}$ be a fixed sequence of complex numbers with
\begin{equation}
  \varepsilon_k \in \mathbb{S}^1\cup \{0\} \qquad k\ge1,
\end{equation}
where $\mathbb{S}^1$ is the unit circle. We make the convention $\varepsilon_0:=1.$

\begin{defn}[Fake $\mu$'s] \label{def: fakemus}
The \emph{fake M\"obius function} $\mathfrak{f}$ corresponding to $(\varepsilon_k)_{k\ge1}$ is the multiplicative function
\begin{equation}
  \mathfrak{f}:\mathbb{N}\longrightarrow\C
\end{equation}
defined on prime powers by
\begin{equation}
  \mathfrak{f}(p^k) := \varepsilon_k
  \qquad p\ \text{prime}, \quad k\ge1.
\end{equation}
In other words, if $n = \prod_p p^{v_p(n)}$, we define
\begin{equation}
  \mathfrak{f}(n) := \prod_p \varepsilon_{v_p(n)}.  
\end{equation}
\end{defn}

\begin{remark}
If we take 
\begin{equation}
    \varepsilon_1=-1, \qquad \varepsilon_k=0 \quad \forall k\ge2,
\end{equation}
then $\mathfrak{f}(n)$ coincides with the classical M\"obius function $\mu(n)$, hence the name fake $\mu$'s. 
\end{remark}
\begin{remark}
Definition \ref{def: fakemus} generalizes the definition introduced in \cite{martin2023fake} for $(\varepsilon_k)_{k\ge1} \subseteq \{0, \pm 1\}.$
\end{remark}
\subsection{The Dirichlet series of fake $\mu$'s}
Given a fake mu $\mathfrak{f}$,  the associated Dirichlet series is
\begin{equation}
  F_{\mathfrak{f}}(s):=  \sum_{n\ge1}\frac{\mathfrak{f}(n)}{n^{s}}, \qquad \Re(s)>1.
\end{equation}
Since $|\mathfrak{f}(n)|\le1$, the series converges absolutely in $\Re(s)>1$. Thus, by absolute convergence of $F_{\mathfrak{f}}(s)$ and multiplicativity of $\mathfrak{f}$, we have the Euler product representation 
\begin{equation}\label{eq: Euler prod Ff}
  F_{\mathfrak{f}}(s)= \prod_{p}\Bigl(1+\varepsilon_1 p^{-s}+\varepsilon_2 p^{-2s}+\varepsilon_3 p^{-3s}+\cdots\Bigr), \qquad \Re(s)>1,
\end{equation}
where the product runs over all primes. 

For analytic continuation beyond $\Re(s)>1$, it is convenient to factor $F_{\mathfrak{f}}(s)$ using complex powers of $\zeta(s)$. Indeed, for $|u|<1$ and $z,w$ in a bounded disk in $\C$, we have the binomial expansion
\begin{equation}\label{eq: f(u) def}
  f(u):=(1-u)^{-z}(1-u^2)^{-w} = 1 + z u + \Bigl(\frac{z(z+1)}{2}+w\Bigr)u^2 + O(u^3).
\end{equation}
Thus, if we choose 
\begin{equation}\label{eq: z w defin}
  z := \varepsilon_1,\qquad
  w := \varepsilon_2 - \frac{\varepsilon_1(\varepsilon_1+1)}{2},
\end{equation}
then the first three coefficients in \eqref{eq: f(u) def} match those of the power series
\begin{equation}
  1+\varepsilon_1 u+\varepsilon_2 u^2+  \varepsilon_3 u^3 +\cdots,
\end{equation}
and we obtain
\begin{equation}\label{eq: f(u) up to 3}
  f(u) = 1+\varepsilon_1 u+\varepsilon_2 u^2 + O(u^3).
\end{equation}

Since $|\varepsilon_j|\le 1$, we know that the power series
\begin{equation}\label{eq: g(u) def}
  g_{\mathfrak{f}}(u):= \sum_{j=0}^\infty \varepsilon_j u^j = 1+\varepsilon_1 u+\varepsilon_2 u^2+\varepsilon_3 u^3+\cdots
\end{equation}
converges absolutely for $|u|<1$.

Now for $\sigma>1,$ we may take $u:= p^{-s}$ and obtain, for each prime $p$, that
\begin{equation}\label{ps}
  1+\varepsilon_1 p^{-s}+\varepsilon_2 p^{-2s}+\varepsilon_3 p^{-3s}+\cdots = (1-p^{-s})^{-z}(1-p^{-2s})^{-w} G_p(s),
\end{equation}
where
\begin{equation}\label{eq: Gp def}
  G_p(s):= (1-p^{-s})^{z}(1-p^{-2s})^{w} g_{\mathfrak{f}}(p^{-s})
\end{equation}
is a holomorphic function on $\sigma>1.$

Consequently, by \eqref{eq: Euler prod Ff} and \eqref{ps}, we have
\begin{equation}
  F_{\mathfrak{f}}(s) = \zeta(s)^{z} \zeta(2s)^{w} G_\mathfrak{f}(s),
\end{equation}
where
\begin{equation}\label{eq:Gf def prod}
  G_\mathfrak{f}(s) := \prod_p G_p(s),    
\end{equation}
and $z,w$ are given by \eqref{eq: z w defin}. Moreover, $G_\mathfrak{f}$ satisfies the following properties.

\begin{prop}[Properties of $G_\mathfrak{f}$]
\label{prop: properties of Gf}
With the same notation as above, we have the following.
\begin{itemize}
    \item[(a)] $G_\mathfrak{f}(s)$ is holomorphic on $\sigma> \frac{1}{3}$.
    \item[(b)] If $\sigma_1>\frac{1}{3}$, then $G_\mathfrak{f}(s)$ is uniformly bounded on $\sigma \ge \sigma_1$.
\end{itemize}
\end{prop}

\begin{proof} 
Fix $\sigma_1>\frac{1}{3}$. Comparing \eqref{eq: f(u) up to 3} and \eqref{eq: g(u) def}, we see that
\begin{equation} 
    h(u) := g_{\mathfrak{f}}(u)-f(u)
\end{equation}
is holomorphic on $|u|<1$ and has a zero of order at least $3$ at $u=0.$ Hence
\begin{equation}
  h(u) = u^3 J(u)
\end{equation}
for some holomorphic function $J$ on $|u|<1$. Therefore, substituting 
\begin{equation}
    g_{\mathfrak{f}}(u) =(1-u)^{-z}(1-u^2)^{-w}+u^3J(u)
\end{equation}
into the definition of $G_p$ in \eqref{eq: Gp def}, we obtain 
\begin{equation}
  G_p(s) -1 = p^{-3s}J(p^{-s})(1-p^{-s})^z(1-p^{-2s})^w .  
\end{equation}
Thus, since $J(p^{-s})$, $(1-p^{-s})^z$, and $(1-p^{-2s})^w$ are uniformly bounded for $\sigma\ge\sigma_1,$
\begin{equation}\label{eq:Gp-1 bound}
  G_p(s) -1  \ll_{\mathfrak{f}, \sigma_1} p^{-3 \sigma_1}   
\end{equation}
uniformly on $\sigma \ge \sigma_1$. 

Thus, since $\sum_p p^{-3\sigma_1}<\infty$ (recall $\sigma_1>1/3$), it follows that the product $G_\mathfrak{f}(s)$ converges absolutely and locally uniformly on $\{\Re(s)>1/3\}$ and defines a holomorphic function there. This proves $(a)$.

For (b), using \eqref{eq:Gp-1 bound}, we know that there is a constant $C_{\sigma_1, \mathfrak{f}}>0$ such that $|G_p(s) -1| \le C_{\sigma_1, \mathfrak{f}}   p^{-3 \sigma_1}$ uniformly on $\sigma \ge \sigma_1$. In particular, $\sum_p |G_p(s) -1|<\infty$ uniformly on $\sigma \ge \sigma_1$, and on $\sigma \ge \sigma_1,$ we have
\begin{align}
  |G_\mathfrak{f}(s)|&=   \bigl| \prod_{p}G_p(s) \bigr| = \bigl|\prod_{p}(1+G_p(s)-1)\bigr| \\
  &\le \prod_{p}(1+|G_p(s)-1|) \le  \prod_{p} \exp(|G_p(s)-1|) \\
  &= \exp\Bigl(\sum_{p}|G_p(s)-1|\Bigr) \le \exp\Bigl(C_{\sigma_1, \mathfrak{f}} \sum_{p}p^{-3\sigma_1}\Bigr),
\end{align}
which is a finite constant depending only on $\sigma_1$ and $\mathfrak{f}$. This proves (b).
\end{proof}

\begin{remark}
    $G_\mathfrak{f}(s)$ may have zeros in the half-plane $\sigma> \frac{1}{3}$.
\end{remark}

\section{Perron formula and the contour}\label{sec: perron form}

For $x>3$ define the exponentially smoothed summatory function
\begin{equation}
A_{\mathfrak{f}}^{\exp}(x) :=\sum_{n\ge1}\mathfrak{f}(n)e^{-n/x}.
\end{equation}
By Mellin inversion \cite[Thms. 1.4 and 1.6]{galway2004analytic} and absolute convergence for $c>1$, we obtain
\begin{equation}\label{eq: perron exp}
A_{\mathfrak{f}}^{\exp}(x) =\frac{1}{2\pi i} \int_{c-i\infty}^{c+i\infty} F_{\mathfrak{f}}(s)\Gamma(s)x^s ds .
\end{equation}
It is convenient to use this smoothing since the exponential decay introduced by the factor $\Gamma(s)$ allows us to carry out our contour estimates. Indeed, Stirling's formula \cite[Eq. (4.12.2)]{titchmarsh1986theory} gives, for
$\sigma_1\le \sigma\le \sigma_2$ with $\sigma_1,\sigma_2 \in \R$ fixed, 
\begin{equation}\label{eq: stirling formula.}
\Gamma(\sigma+it) \ll |t|^{\sigma-1/2}e^{-\pi |t|/2}, \qquad\qquad |t|\to\infty .
\end{equation}

Fix
\begin{equation}
\frac13<a<\frac{1}{2},\qquad c>1.
\end{equation}
For the remainder of the paper, we will assume $T$ satisfies \eqref{eq: height T lower bd}, and in particular may be taken arbitrarily large. We truncate the line of integration in
\eqref{eq: perron exp} at height $T$ and shift it from $\Re(s)=c$ to $\Re(s)=a$ using the contour $\Gamma_T$ shown in Figure \ref{fig:contour}. We construct $\Gamma_T$ as a counterclockwise rectangle
with vertices $c\pm iT$ and $a\pm iT$, modified by Hankel contours $H(\xi)$ around the
branch points
\begin{equation}
\xi=1,\qquad \xi=\frac{1}{2},\qquad \xi=\rho
\quad (|\Im\rho|<T).
\end{equation}

\begin{figure}
    \centering
    \includegraphics[width=0.84\linewidth]{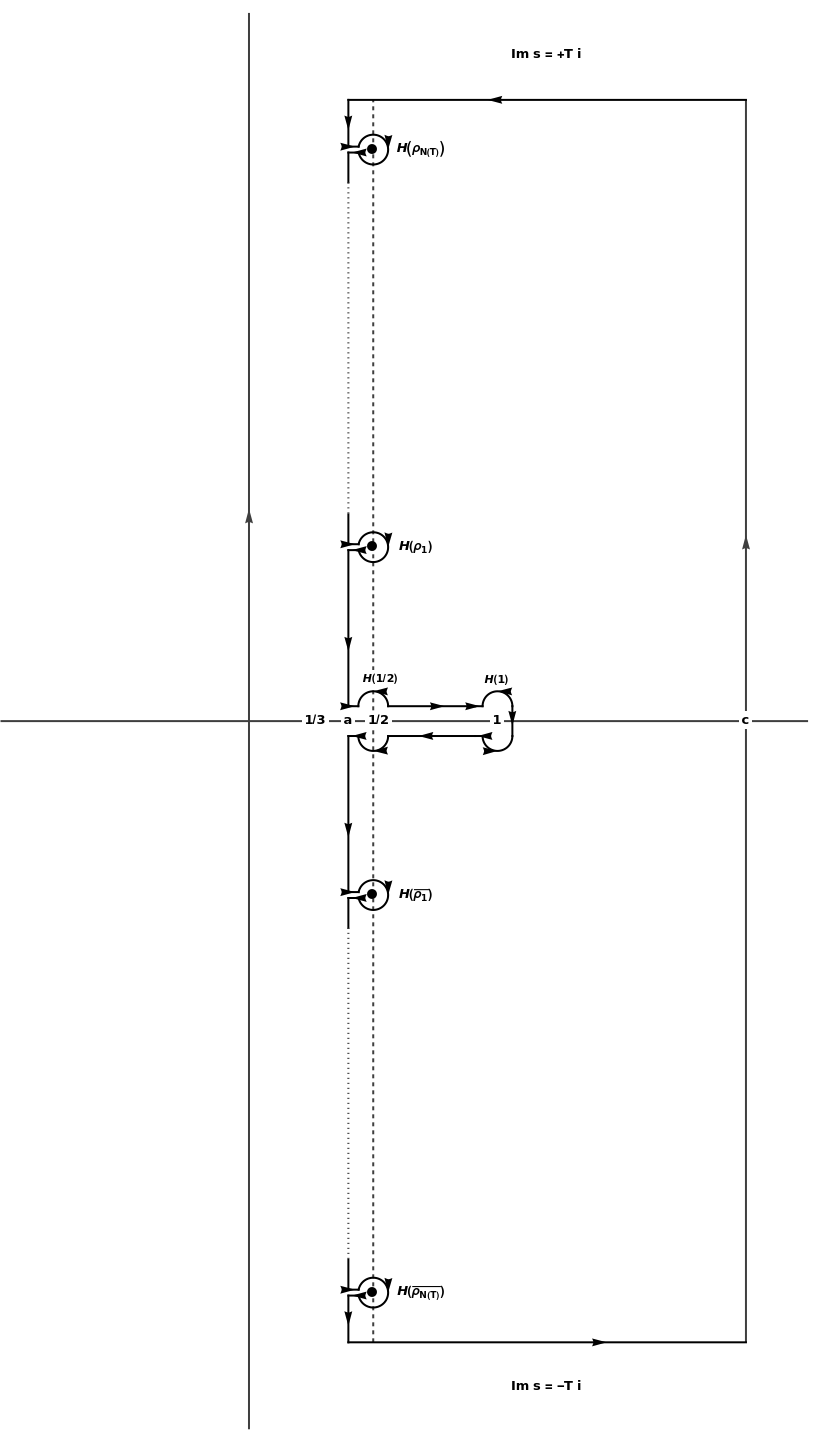}
    \caption{The contour $\Gamma_T$.}
    \label{fig:contour}
\end{figure}

We shall use the following notation for the kernel
\begin{equation} \label{eq: psi}
    \Psi(s;x):= \Gamma(s)   x^s, \qquad s \in \C \setminus \Z_{\le0} , \quad x\ge 1.
\end{equation}
Fix $r_0=r_0(T)>0$ small enough so that $r_0<\tfrac12-a$ and, for every zero $\rho=\tfrac12+i\gamma$ with $|\gamma|<T$, the disk $|s-\rho|<r_0$ contains no zero of $\zeta(s)$ other than $\rho$. Throughout, let $0<r<r_0.$ For each branch point $\xi \in \{1\} \cup \{\rho : |\gamma|<T\} ,$ write 
\begin{equation} \label{eq: H(xi) parts}
    H(\xi)=:I_{\xi,r}^+\cup C(\xi,r)\cup I_{\xi,r}^-,
\end{equation}
where $C(\xi,r)$ is the circular arc of radius $r$ centered at $\xi$, and $I_{\xi,r}^{\pm}$ are the corresponding horizontal segments \protect \footnote{Geometrically, $C(\xi,r)$ intersects the segments $I_{\xi,r}^{\pm}$ at the points $\xi- \frac{\sqrt{3}}{2}r \pm i\frac{r}{2},$ respectively.}  just above/below the branch cut at $\xi$, at a distance $r/2$. The segment $I_{\xi,r}^{+}$ is oriented left-to-right, while  $I_{\xi,r}^{-}$ is oriented right-to-left. Similarly, for $\xi=1/2,$ we write 
\begin{equation}
    H(1/2)=:I_{1/2,r}^+\cup  I_{1/2,r}^-.
\end{equation}

To account for the additional contributions coming from vertical segments on the line $\Re(s)=a$, we define
\begin{equation} \label{eq: V1/21}
    V_{1/2,1}:= \{a+i t: r/2 \le |t|\le r_0 \},
\end{equation}
\begin{equation} \label{eq: Vrho}
    V_\rho:= \{a+ i \Im(\rho) + i t: r/2 \le |t|\le r_0 \}.
\end{equation}
Then
\begin{equation}
    A_{\mathfrak{f}}^{\exp}(x)= I_c(x;T) + E_0(x;T),
\end{equation}
where 
\begin{equation}
  I_c(x;T):= \frac{1}{2\pi i}\int_{c-iT}^{c+iT} F_{\mathfrak{f}}(s)  \Psi(s;x)  ds,
\end{equation}
and
\begin{equation} \label{eq: E0}
    E_0(x;T)= \frac{1}{2\pi i}\int_{c-i \infty}^{c-iT}  F_{\mathfrak{f}}(s)   \Psi(s;x) ds + \frac{1}{2\pi i}\int_{c+i T}^{c+i \infty}  F_{\mathfrak{f}}(s)   \Psi(s;x) ds.
\end{equation}
On the other hand, by Cauchy's theorem, we have
\begin{equation}
    \int_{\Gamma_T} F_{\mathfrak{f}}(s)   \Psi(s;x)  ds=0.
\end{equation}
Thus, it follows that
\begin{align}
 A_{\mathfrak{f}}^{\exp}(x) &= \tilde \Delta_1(x)
  + \tilde \Delta_{1/2}(x) + \delta_{1/2,1}(x) +  \sum_{\rho :   |\Im(\rho)|<T} \Bigl(  \tilde \Delta_\rho(x)  + \delta_{\rho}(x)  \Bigr) \\
  & \quad+ E_0(x;T) + E_1(x; T) + E_2(x; T), \label{eq: Afm decomposition}
\end{align}
where, for $\xi\in\{1,\tfrac12\}\cup\{\rho:\ |\Im\rho|<T\}$, 
\begin{equation}\label{eq: tilde Delta xi def}
  \tilde \Delta_\xi(x) := -\frac{1}{2\pi i}\int_{H(\xi)} F_{\mathfrak{f}}(s)   \Psi(s;x) ds,
\end{equation}
\begin{equation}\label{eq: deltahalf1}
    \delta_{1/2,1}(x):=- \frac{1}{2\pi i} \int_{V_{1/2,1}} F_{\mathfrak{f}}(s)   \Psi(s;x) ds,
\end{equation}
\begin{equation}\label{eq: deltarho}
    \delta_{\rho}(x):=-\frac{1}{2\pi i} \int_{V_{\rho}} F_{\mathfrak{f}}(s)   \Psi(s;x) ds,
\end{equation}
\begin{equation} \label{eq: E1}
    E_1(x;T):= -\frac{1}{2\pi i} \bigl(\int_{c+iT}^{a+iT} F_{\mathfrak{f}}(s)  \Psi(s;x) ds + \int_{a-iT}^{c-iT} F_{\mathfrak{f}}(s) \Psi(s;x) ds \bigr), 
\end{equation}
and
\begin{equation} \label{eq: E2}
    E_2(x;T):= -\frac{1}{2\pi i}\int_{V_{a,T}} F_{\mathfrak{f}}(s)   \Psi(s;x) ds, 
\end{equation}
with 
\begin{equation} \label{def VaT}
V_{a,T} := \{s: \Re(s) = a,  |\Im s| \le T\} \setminus \Bigl(  V_{1/2,1}   \cup   \bigcup_{ |\gamma| \le T} V_{\rho} \Bigr).    
\end{equation}

\section{Proofs of the main results} \label{sec: proofs}

\subsection{Proof of Theorem \ref{thm: exp form Af}}
\subsubsection{The integrands} \label{sec: the integrands}
The integrands in Theorem \ref{thm: exp form Af} are as follows. 
\begin{equation}\label{J1}
\mathcal{J}_1(u;\mathfrak{f}):= (1-u)  Z(1-u;z) \zeta(2-2u)^{w}  G_{\mathfrak{f}}(1-u)   \Gamma(1-u),  
\end{equation}
\begin{equation}\label{J12}
\mathcal{J}_{1/2}(u;\mathfrak{f}):= \Bigl(\tfrac12-u\Bigr)^2 \Bigl(\tfrac12+u\Bigr)^{-z} Z\Bigl(\tfrac12-u;z\Bigr)  Z\bigl(1-2u;w\bigr)  G_{\mathfrak{f}}\Bigl(\tfrac12-u\Bigr)   \Gamma \Bigl(\tfrac12-u\Bigr),
\end{equation}
and for a zero $\rho$ of $\zeta(s),$
\begin{equation}\label{Jrho}
    \mathcal{J}_\rho(u;\mathfrak{f}):= (\rho-1-u)^{-z} G_{\mathfrak{f}}(\rho-u)  Z_\rho(\rho-u;z)  \zeta(2\rho-2u)^{w}  \Gamma(\rho-u).
\end{equation}

\subsubsection{The proof}
\begin{proof}
Fix $1/3<a<1/2$ and $c>1.$ We may assume, without loss of generality, that $c<2.$ For convenience, set
\begin{equation}
     \widetilde F(s;x):=  F_{\mathfrak{f}}(s)   \Psi(s;x).
\end{equation}
Recall the decomposition of $A_{\mathfrak{f}}^{\exp}(x)$ into the main and error terms in \eqref{eq: Afm decomposition}-\eqref{eq: E2}. It remains to prove the following two claims.
\begin{itemize}
  \item[(a)] There exists a constant $T_0=T_0(a,c,\mathfrak{f})>0$ such that for every $x \ge 3$ and every $T\ge T_0$ satisfying \eqref{eq: height T lower bd}, the error terms $E_0(x;T),E_1(x;T),$ and $E_2(x;T),$ given in \eqref{eq: E0}, \eqref{eq: E1}, and \eqref{eq: E2}, satisfy 
  \begin{equation}
      E_0(x;T)=O(x^c e^{-T}),\qquad E_1(x;T)=O(x^c e^{-T}),\qquad E_2(x;T)=O(x^a).
  \end{equation}
    Hence, if $T>(c-a)\Log x$ and $x$ is large, it follows that the total error $E_0(x;T)+E_1(x;T)+E_2(x;T)$ is $O(x^a).$
  \item[(b)] The main contributions, given in  \eqref{eq: tilde Delta xi def}, \eqref{eq: deltahalf1}, and \eqref{eq: deltarho}, add up to the Laplace integrals $\Delta_\xi(x)$ stated in Theorem \ref{thm: exp form Af}. In particular, we have 
  \begin{equation}
      \tilde \Delta_1(x) + \tilde \Delta_{1/2}(x) + \delta_{1/2,1}(x) = \Delta_1(x) +  \Delta_{1/2}(x),
  \end{equation}
  and for every zero $\rho$ with $|\Im \rho| < T$
  \begin{equation}
      \tilde \Delta_\rho(x) +\delta_\rho(x)= \Delta_\rho(x).
  \end{equation}
\end{itemize}

We begin by proving $(a).$ First, we will need a bound on the integrand $\widetilde F(s;x).$

\begin{lemma}\label{lem:F growth}
With the notation above, the following uniform estimates hold.

\begin{enumerate}
\item There exists $t_0=t_0(c,\mathfrak{f})>0$ such that, for every $x\ge 3$ and every real $t$ with
$|t|\ge t_0$,
\begin{equation} \label{eq: F growth 1}
    \widetilde F(c+it;x) \ll_{a,c,\mathfrak{f}} x^c e^{-|t|}.
\end{equation}

\item Recall the set $V_{a,T}$ in \eqref{def VaT}. For $s=a+it\in V_{a,T}$,
\begin{equation} \label{eq: F growth 2}
    \widetilde F(s;x) \ll_{a,\mathfrak{f}} x^a e^{-|t|}.
\end{equation}

\item There exists $T_0=T_0(a,c,\mathfrak{f})>0$ such that, for every $x\ge 3$ and every
$T\ge T_0$ satisfying \eqref{eq: height T lower bd},
\begin{equation} \label{eq: F growth 3}
    \widetilde F(\sigma\pm iT;x) \ll_{a,c,\mathfrak{f}} x^\sigma e^{-T}
\end{equation}
uniformly for $a\le \sigma\le c$.
\end{enumerate}
\end{lemma}

\begin{remark}
    The estimate \eqref{eq: F growth 1} is used to bound the contributions of the tails at the vertical line $\sigma=c,$ and \eqref{eq: F growth 2} is used to bound the contribution of the punctured vertical line $V_{a,T}$ at $\sigma=a.$ We use \eqref{eq: F growth 3} to estimate the contribution of the horizontal lines at height $\pm T.$
\end{remark}

\begin{proof}
    We first prove \eqref{eq: F growth 1} and \eqref{eq: F growth 2}. By Proposition \ref{prop: properties of Gf}, we know $G_{\mathfrak{f}}$ is uniformly bounded on $\Re(s)\ge a$, and thus bounding complex powers of $\zeta(s)$ using the estimates in Lemma \ref{lem:zeta z bounds}, we know that there exists a constant $B>0$ depending only on $a,c$ and $\mathfrak{f}$ such that
    \begin{equation}
        \zeta(s)^z \zeta(2s)^w G_{\mathfrak{f}}(s)  \ll_{a,c,\mathfrak{f}}  (1+|t|)^B,
    \end{equation}
    uniformly for $s=c+it$ with $|t| \ge 3$, and also for $s$ in $\{|t|\ge2\}  \cap  V_{a,T}$.

    Thus, choosing $t_0 \ge 3$ large enough (possibly depending on $B$) so that Stirling's formula \eqref{eq: stirling formula.} applies and the polynomial factor is controlled by the exponential decay, we obtain 
    \begin{align}
      \bigl|\widetilde F(s;x)\bigr| &= \bigl|\zeta(s)^z \zeta(2s)^w G_{\mathfrak{f}}(s) \Gamma(s) x^{s}\bigr|
      \ll x^\sigma (1+|t|)^{B} |t|^{\sigma - 1/2} e^{-\pi|t|/2} \nonumber \\
      &= x^\sigma (1+|t|)^{B} |t|^{\sigma - 1/2}  e^{-(\pi/2 -1)|t|} e^{-|t|} \ll x^\sigma e^{-|t|},
    \end{align}
    uniformly for $s=c+it$ with $|t| \ge t_0$, and uniformly for $s$ in $\{|t|\ge t_0 \}  \cap  V_{a,T}$. This proves \eqref{eq: F growth 1}. For \eqref{eq: F growth 2}, it remains to prove the required estimate inside the box $|t|\le t_0.$
    
    We show that $\zeta(s)^z \zeta(2s)^w G_{\mathfrak{f}}(s) \Gamma(s)$ is bounded on $\{|t| \le t_0\}\cap V_{a,T}$. Then $\widetilde F(s;x)= O(x^a)=O(x^a e^{-|t|})$ there, proving \eqref{eq: F growth 2}. Now $G_{\mathfrak{f}}(s) \Gamma(s)$ is continuous, hence bounded, on $\{a+it: |t| \le t_0\}$ and $\zeta(s)^z \zeta(2s)^w$ is continuous and bounded on the line $\sigma=a$ away from the branch cuts. Hence, it remains to bound $\zeta(s)^z \zeta(2s)^w$ near the points where the line $\sigma=a$ crosses the branch cuts, namely at $s=a$ and $s=a + i \gamma$ with $|\gamma|\le t_0,$ if any. 
    
    Since $1/3<a<1/2$, the local factorizations \eqref{zetasdef},
    \begin{equation}
      \zeta(s)^z = s(s-1)^{-z} Z(s;z),
      \qquad
      \zeta(2s)^w = 2s(2s-1)^{-w} Z(2s;w),
    \end{equation}
    show that the upper and lower boundary values at $s=a$ are finite. Indeed, 
    \begin{align}
      \lim_{r\to0^+}
        \zeta(a \pm i r/2)^z 
        \zeta\bigl(2a \pm ir\bigr)^w
      &= 2a^2 Z(a;z) Z(2a;w)  \lim_{r\to0^+}
           (a \pm i r/2-1)^{-z}
           \bigl(2a \pm ir-1\bigr)^{-w} \nonumber \\
      &= 2 a^2 Z(a;z) Z(2a;w)  e^{\mp i\pi(z+w)}  
         (1-a)^{-z}(1-2a)^{-w} \nonumber \\
      &= O_{a,\mathfrak{f}}(1).
    \end{align}
    Similarly, using the local factorizations of $\zeta(s)^z$ near the zeros of $\zeta(s)$, as in Corollary \ref{cor:localform}, the upper and lower limits of $\zeta(s)^z \zeta(2s)^w$ at each point $a+i \gamma$ with $|\gamma| \le t_0$, if any,  exist and are finite. This proves \eqref{eq: F growth 2}.

    Estimate \eqref{eq: F growth 3} is proved similarly. By Lemma \ref{lem:zeta z bounds} and Proposition \ref{prop: properties of Gf}, there exists a constant $C>0$ depending only on $a,c$ and $\mathfrak{f},$ such that 
    \begin{equation}
        \zeta(s)^z\zeta(2s)^wG_{\mathfrak{f}}(s) \ll_{a,c,\mathfrak{f}} T^{C\Log\Log T},
    \end{equation}
    uniformly for $s=\sigma \pm i T$ with $a\le\sigma\le c$. The last bound is subexponential of the form $e^{o(T)}$ as $T \to \infty.$ Thus, using Stirling's formula, we get the desired bound \eqref{eq: F growth 3} for large enough $T$ satisfying \eqref{eq: height T lower bd}.
\end{proof}

\begin{corollary}\label{errorsbound}
    Let $E_0,$ $E_1,$ and $E_2$ be as in \eqref{eq: E0}, \eqref{eq: E1}, and \eqref{eq: E2}. Then, with the above setup, for every $x\ge3$ and every $T\ge T_0$ satisfying \eqref{eq: height T lower bd}, we have
    \begin{equation}\label{errorsboundeq}
        E_0(x;T) = O\bigl(x^c e^{-T}\bigr),\qquad
        E_1(x;T) = O\bigl(x^c e^{-T}\bigr),\qquad
        E_2(x;T) = O\bigl(x^a \bigr),
    \end{equation}
    where the implied constants may depend on $a,c,\mathfrak{f}$, but not on $x$ or $T$.
\end{corollary}

\begin{proof}
By Lemma \ref{lem:F growth}, it follows that, for such $T\ge T_0$, we have
\begin{align}
  |E_0(x;T)|
  &\ll_{a,c, \mathfrak{f}}
    x^c\int_T^\infty e^{- t} dt   =x^c e^{-T} \int_T^\infty e^{- (t-T)} dt  \\
    &=x^c e^{-T} \int_0^\infty e^{- t} dt  \ll 
    x^c e^{-T},
\end{align}

\begin{align}
    |E_1(x;T)|&\ll  \int_a^c |\widetilde F(\sigma+iT;x)| d\sigma +\int_a^c |\widetilde F(\sigma-iT;x)| d\sigma \\
    &\ll_{a,c,\mathfrak{f}} e^{-T}\int_a^c x^\sigma d\sigma \ll_{a,c,\mathfrak{f}} x^c e^{-T},
\end{align}
and
\begin{align}
    |E_2(x;T)| &\ll \int_{\{t: a+it\in V_{a,T}\}} |\widetilde F(a+it;x)| dt  \\
    &\ll_{a,\mathfrak{f}} x^a \int_{-T}^{T} e^{-|t|} dt \ll x^a .
\end{align}
This completes the proof.
\end{proof}

This establishes $(a).$ We now turn to the proof of $(b).$ For a horizontal segment $I$, oriented left to right in the complex plane, we define
\begin{equation}
  \Delta_I(x)
:= -\frac{1}{2\pi i}\int_I
       \bigl(\widetilde F_+(s;x)-\widetilde F_-(s;x)\bigr) ds,  
\end{equation}
where
\begin{equation}
    \widetilde F_\pm(s;x):= \widetilde F(s \pm 0 i;x):= \lim_{\varepsilon \downarrow 0} \widetilde F(s \pm \varepsilon i;x), \qquad s\in I,
\end{equation}
denote the boundary values above and below the slit $I$.

We will show, by Cauchy's theorem, that the contributions of the Hankel contours $H(\xi)$ can be replaced by the corresponding branch cut contributions:

\begin{equation} \label{eq: b.1}
     \tilde \Delta_1(x) + \tilde \Delta_{1/2}(x) + \delta_{1/2,1}(x) = \Delta_{I_1}(x) +  \Delta_{I_{1/2}}(x)  
\end{equation}
and for every zero $\rho$ with $|\Im \rho|<T$,
\begin{equation} \label{eq: b.2}
\tilde \Delta_\rho(x) +\delta_\rho(x)= \Delta_{I_\rho}(x),
\end{equation}
where
\begin{equation}
   I_1=(\tfrac12,1),\qquad
I_{1/2}=(a,\tfrac12),\qquad
I_\rho=(a,\tfrac12)+i\Im\rho. 
\end{equation}

Recall $\tilde \Delta_\xi,$ $\delta_{1/2,1},$ and $\delta_\rho$ are defined in \eqref{eq: tilde Delta xi def}, \eqref{eq: deltahalf1}, and \eqref{eq: deltarho}. Once this last claim is established, we will then show that $\Delta_\xi(x)= \Delta_{I_\xi}(x)$ for every $\xi,$ where $\Delta_\xi(x)$ are the Laplace integrals appearing in Theorem \ref{thm: exp form Af}. This will thus prove $(b)$ and hence the theorem.

To prove \eqref{eq: b.1} and \eqref{eq: b.2}, we begin by showing that the contributions of the circular arc part $C(\xi,r)$ of each Hankel piece $H(\xi)$ vanishes in the limit. Recall Figure \ref{fig:contour}.

\begin{lemma}\label{circlesvanish}
  Fix $x>3$. Assume $-1<\Re z<1$. Then for
  $\xi \in \{1\} \cup \{\rho : |\gamma| < T\},$ we have
  \begin{equation}
    \lim_{r\to0^+} \int_{C(\xi,r)} \widetilde F(s;x) ds  =  0.
  \end{equation}
  Here $C(\xi,r)$ denotes the circular arc of radius $r$ centered at $\xi$ that appears in the Hankel contour $H(\xi).$ See \eqref{eq: H(xi) parts} and Figure \ref{fig:contour}.
\end{lemma}

\begin{proof}
We begin with the case $\xi=1$. Since $\zeta(2s)^w$, $G_{\mathfrak{f}}(s)$, $\Gamma(s)$, and $x^s$ are all
holomorphic near $s=1$, we observe, by the factorization \eqref{eq: zetazdef} of $\zeta(s)^z$ near $s=1$, that
\begin{align}
    \bigl|\widetilde F(s;x)\bigr| &= \bigl|\zeta(s)^z \zeta(2s)^w G_{\mathfrak{f}}(s) \Gamma(s) x^{s}\bigr| \ll_{\mathfrak{f},x} \bigl|s (s-1)^{-z} Z(s;z)\bigr| \ll_{\mathfrak{f},x }|s-1|^{-\Re z} \label{eq: bound Ff near 1}
\end{align}
for $s$ on a sufficiently small circle $C(1,r)$.
Thus, 
\begin{equation}
  \bigl |\int_{C(1,r)} \widetilde F(s;x) ds\bigr |
   \ll 
  r \cdot r^{-\Re z}
   = 
  r^{ 1-\Re z}.
\end{equation}
Since $\Re z<1$, we have $r^{1-\Re z}\to0$ as $r\to0^+$, proving the claim for $\xi=1.$ 

\smallskip
Now let $\rho=\tfrac12+i\gamma$ be a zero with $|\gamma|<T$. By the factorization \eqref{eq: zeta z near rho} of $\zeta(s)^z$ near $s=\rho$, we similarly have
\begin{equation} \label{eq: bound of Ff near rho}
  \bigl|\widetilde F(s;x)\bigr| \ll_{\rho, \mathfrak{f},x} |s-\rho|^{\Re z}
\end{equation}
for $s$ on a sufficiently small circle $C(\rho,r)$. The claim thus follows since $\Re z>-1$.
\end{proof}

We now prove \eqref{eq: b.1} and \eqref{eq: b.2}.
\begin{lemma}
Fix $x>3$ and assume the same notation as above. Assume also that 
\begin{equation}\label{eq:absconv}
-1<\Re z<1,\qquad \Re(w)<1.
\end{equation}
Then \eqref{eq: b.1} and \eqref{eq: b.2} hold.
\end{lemma}

\begin{proof}
Then by Cauchy's theorem, we have
\begin{align}
    &\int_{H(1)}\widetilde F ds + \int_{H(1/2)}\widetilde F ds+ \int_{V_{1/2,1}}\widetilde F ds \\
&=   \lim_{r \to0^+} \Bigl( \int_{I_{1/2,r}^+}\widetilde F ds+\int_{I_{1/2,r}^-}\widetilde F ds\Bigr) + \lim_{r \to0^+} \Bigl( \int_{I_{1,r}^+}\widetilde F ds+\int_{C(1,r)}\widetilde F ds+\int_{I_{1,r}^-}\widetilde F ds\Bigr),
\end{align}
and for every zero $\rho$ with $|\Im\rho|<T$,
\begin{equation}
\int_{H(\rho)}\widetilde F ds+ \int_{V_\rho}\widetilde F ds
=   \lim_{r \to0^+} \Bigl( \int_{I_{\rho,r}^+}\widetilde F ds+\int_{C(\rho,r)}\widetilde F ds+\int_{I_{\rho,r}^-}\widetilde F ds\Bigr),
\end{equation}
where $V_{1/2,1}$ and $V_\rho$ are defined in \eqref{eq: V1/21} and \eqref{eq: Vrho}, respectively. 

First, by Lemma \ref{circlesvanish}, we have  $\int_{C(\xi,r)}\widetilde F ds \to 0$ as $r\to0^+$  for
$\xi\in\{1\}\cup\{\rho: |\gamma|<T\}$.

Moreover, using \eqref{eq: bound Ff near 1} and \eqref{eq: bound of Ff near rho}, we observe that
\begin{equation}\label{eq: tilde F near half}
    |\widetilde F(t \pm i r/2;x)|  \ll  (1-2t)^{-\max(\Re(w),0)}, \qquad t\in(a,\tfrac12),\ 0<r\le r_0,
\end{equation}
\begin{equation}\label{eq: tilde F near rho}
    |\widetilde F(\rho-t \pm i r/2;x)|  \ll_\rho  t^{\min(\Re z,0)}, \qquad t\in(0,\tfrac12-a),\ 0<r\le r_0,
\end{equation}
and for $t\in(\tfrac12,1)$ and $0<r\le r_0$,
\begin{align}\label{eq: tilde F near 1}
|\widetilde F(t \pm i r/2;x)|  &\ll  (1-t)^{-\max(\Re z,0)} (2t-1)^{-\max(\Re(w),0)}.
\end{align}

Now since $-1<\Re(z)<1$ and $\Re(w)<1,$ the dominating functions in \eqref{eq: tilde F near half}, \eqref{eq: tilde F near rho}, and \eqref{eq: tilde F near 1} are integrable on the respective intervals. The claim thus follows by the dominated convergence theorem. This completes the proof. 
\end{proof}

To finish the proof of $(b)$, it thus suffices to compute the jumps $\widetilde F_+-\widetilde F_-$ across the branch cuts. In what follows, we will use the local factorizations \eqref{eq: zetazdef} and \eqref{eq: zeta z near rho} of complex powers of $\zeta(s)$ near $s=1$ and $s=\rho$.

First consider $I_1=(\tfrac12,1)$. On $I_1$, the boundary values of $(t-1)^{-z}$ are
\begin{equation}
(t-1)^{-z}\big|_{\pm}=(1-t)^{-z}e^{\mp i\pi z},
\end{equation}
and therefore writing $\zeta(s)^z= s (s-1)^{-z}Z(s;z),$ we obtain
\begin{equation}
\widetilde F_+(t;x)-\widetilde F_-(t;x)
=
-2i\sin(\pi z)(1-t)^{-z}
tZ(t;z)\zeta(2t)^wG_{\mathfrak{f}}(t)\Gamma(t)x^t.
\end{equation}
It follows that
\begin{align}
\Delta_{I_1}(x)=\frac{\sin(\pi z)}{\pi} \int_{1/2}^{1} (1-t)^{-z}tZ(t;z)\zeta(2t)^wG_{\mathfrak{f}}(t)\Gamma(t)x^t dt.
\end{align}
Using $u=1-t$, this becomes
\begin{equation}
\Delta_{I_1}(x)= \frac{\sin(\pi z)}{\pi}x \int_0^{1/2}e^{-Lu}u^{-z}\mathcal J_1(u;\mathfrak{f}) du,
\end{equation}
where $\mathcal J_1(u;\mathfrak{f})$ is exactly \eqref{J1}. Thus $\Delta_{I_1}(x)=\Delta_1(x)$.

The computations on $I_{1/2}=(a,\tfrac12)$ and $I_\rho=(a,\tfrac12)+i\Im\rho$  with $|\gamma|<T$ are symmetric. On $I_{1/2}=(a,\tfrac12)$, we use
\begin{equation}
\widetilde F_\pm(t;x)
=\Bigl((t-1)^{-z}(2t-1)^{-w}\Bigr)\Big|_\pm
2t^2Z(t;z)Z(2t;w)G_{\mathfrak{f}}(t)\Gamma(t)x^t,
\end{equation}
with
\begin{equation}
(t-1)^{-z}\big|_\pm=(1-t)^{-z}e^{\mp i\pi z},
\qquad
(2t-1)^{-w}\big|_\pm=(1-2t)^{-w}e^{\mp i\pi w}.
\end{equation}

On $I_\rho=(a,\tfrac12)+i\gamma$, write $s=\rho-u$ with $u\in(0,\tfrac12-a),$ so that
\begin{equation}
\begin{aligned}
\widetilde F_\pm(\rho-u;x)
&=(s-\rho)^z\big|_\pm
(\rho-1-u)^{-z}Z_\rho(\rho-u;z)
\zeta(2\rho-2u)^w  \\
&\qquad\times
G_{\mathfrak{f}}(\rho-u)\Gamma(\rho-u)x^{\rho-u},
\end{aligned}
\end{equation}
with
\begin{equation}
(s-\rho)^z\big|_\pm=(-u)^z\big|_\pm=u^z e^{\pm i\pi z}.
\end{equation}

After suitable changes of variables, we recover $\Delta_{I_\xi}(x)=\Delta_\xi(x)$ for all $\xi \in \{1,1/2\} \cup \{\rho: |\gamma|<T\}$. By their definitions in Section \ref{sec: the integrands}, the functions $\mathcal J_1$, $\mathcal J_{1/2}$, and $\mathcal J_\rho$ are clearly holomorphic near $u=0$. This completes the proof of Theorem \ref{thm: exp form Af}. 
\end{proof}

\subsection{Proof of Theorem \ref{thm: watson exp for Delta}}

\begin{proof}
We use Watson's lemma \cite[p. 53-57]{miller2006applied}. Let $v>0$, and suppose
$\mathcal K:(0,v)\to\C$ is absolutely integrable. Suppose also that, for some
$s\in\C$ with $\Re(s)>-1$, we have
\begin{equation}
    \mathcal K(u)=u^s\mathcal J(u)
\end{equation}
near $u=0$, where $\mathcal J$ has $M+1$ continuous derivatives near $0$. Then Watson's lemma gives, as
$L\to\infty$,
\begin{equation}\label{eq: Watson lemma}
    \int_0^v e^{-Lu}\mathcal K(u) du= \sum_{k=0}^{M} \frac{\mathcal J^{(k)}(0)}{k!}  \Gamma(s+k+1)L^{-(s+k+1)} + O_M\bigl(L^{-(\Re(s)+M+2)}\bigr).
\end{equation}

We apply \eqref{eq: Watson lemma} to the Laplace integrals $\Delta_\xi(x)$ given in
Theorem \ref{thm: exp form Af}. The corresponding kernels are absolutely integrable:
\begin{equation}
    u^{-z}\mathcal J_1(u;\mathfrak{f})\in L^1(0,\tfrac12),
    \qquad
    u^{-w}\mathcal J_{1/2}(u;\mathfrak{f})\in L^1(0,\tfrac12-a),
    \qquad
    u^z\mathcal J_\rho(u;\mathfrak{f})\in L^1(0,\tfrac12-a).
\end{equation}
These follow from $\Re z<1$, $\Re(w)<1$, and
$\Re z>-1$, respectively. For completeness, we note that 
$\mathcal J_1$ has a singularity at $u=\tfrac12$, coming from the factor
\begin{equation}
    \zeta(2-2u)^w=(2-2u)(1-2u)^{-w}Z(2-2u;w).
\end{equation}
This singularity is integrable because $\Re(w)<1$.

Expanding each $\mathcal J_\xi(u;\mathfrak{f})$ in a power series near $u=0,$ the claim follows.
\end{proof}

\subsection{Proof of Theorem \ref{thm: bias theorem}}

\subsubsection{The constant $c_{1/2}(\mathfrak{f};z,w)$} \label{sec: cst c12}
Recall that, in \eqref{eq: c12 def}, we defined
\begin{equation}
    c_{1/2}(\mathfrak{f};z,w)= \frac{\sin\bigl(\pi(z+w)\bigr)}{\pi} 2^{-w+1}  \Gamma(1-w) \lambda_{1/2,0}(\mathfrak{f};z,w).
\end{equation}
Using the definitions $\lambda_{1/2,0}=\mathcal J_{1/2}(0;\mathfrak{f})$ and \eqref{J12}, we obtain for $z, w \not \in \Z$ with $-1<\Re(z)<1$ and $\Re(w)<1,$
\begin{equation}\label{eq: expl c12}
    c_{1/2}(\mathfrak{f};z,w)= \frac{2^{-w}}{\sqrt{\pi}} \sin\bigl(\pi(z+w)\bigr) \Gamma(1-w) \bigl(-\zeta(1/2) \bigr)^z G_\mathfrak{f}(1/2),
\end{equation}
where $G_\mathfrak{f}$ is defined in \eqref{eq:Gf def prod}. From \eqref{eq: expl c12}, it follows that 
\begin{equation}
    c_{1/2}(\mathfrak{f};z,w)=0 \qquad \iff \qquad z+w \in \Z, \quad \text{or} \quad G_\mathfrak{f}(1/2)=0.
\end{equation}
Now from \eqref{eq: Gp def}, \eqref{eq:Gf def prod}, and absolute convergence of $G_\mathfrak{f}(1/2)$, we see that $G_\mathfrak{f}(1/2)=0$ iff $g_\mathfrak{f}(p^{-1/2})=0$ for some prime $p,$ where $g_{\mathfrak{f}}(u)= \sum_{j=0}^\infty \varepsilon_j u^j$. However, since $|\varepsilon_j|\le1,$ it follows that 
\begin{equation}
    \Bigl| \sum_{j=1}^\infty \varepsilon_j p^{-j/2} \Bigr| \le \sum_{j=1}^\infty  p^{-j/2} =\frac{p^{-1/2}}{1- p^{-1/2}}.
\end{equation}
Therefore, $g_\mathfrak{f}(p^{-1/2})=0$ only if $\frac{p^{-1/2}}{1- p^{-1/2}} \ge 1,$ so that $p=2,3.$ Thus, we have the following criterion 
\begin{equation}
    c_{1/2}(\mathfrak{f};z,w)=0 \qquad \iff \qquad z+w \in \Z, \quad \text{or} \quad g_\mathfrak{f}(2^{-1/2}) g_\mathfrak{f}(3^{-1/2})=0  
\end{equation}

\begin{ex}
    Let $\varepsilon_k= e^{\frac{k\pi i}{3}}$ for $k\ge 1.$ Then 
    \begin{equation}
        g_\mathfrak{f}(p^{-1/2})=\frac{1}{1-e^{\frac{\pi i}{3}} p^{-1/2}} \neq 0, 
    \end{equation}
    for all primes $p$. Moreover, $z+w = i \sqrt{3}/2 \not \in \Z.$ Thus, $c_{1/2}(\mathfrak{f};z,w) \neq0.$ In this example, $\Re(z)=1/2,$ $\Re(w)=-1/2,$ and $\Re(z+w)=0,$ so our bias theorem applies and classifies this example as an apparent bias.
\end{ex}

\subsubsection{The WMH assumption} \label{sec: WMH}
In Theorem \ref{thm: bias theorem}, we assumed the weak Mertens hypothesis (WMH):
\begin{equation}\label{eq:WMH}
    \int_1^X\bigl(\frac{M(x)}{x}\bigr)^2 dx \ll \Log X.
\end{equation}

A useful consequence of \eqref{eq:WMH} is a non-trivial lower bound for gaps between consecutive ordinates of zeta zeros: under WMH, there exists a constant $A_1>0$ such that, for all consecutive ordinates
$e^e<\gamma<\gamma'$ of zeros $\rho=\tfrac12+i\gamma$ and $\rho'=\tfrac12+i\gamma'$, we have
\begin{equation}\label{eq: consecutive zero spacing}
    \gamma'-\gamma >\  \frac{A_1}{\gamma}\exp \Bigl(-A_1 \frac{\Log \gamma}{\Log \Log \gamma}\Bigr)\ge \frac{A_1}{\gamma^{A_1+1}}  .
\end{equation}
See, for example, Titchmarsh \cite[14.29.1, 14.31]{titchmarsh1986theory}.

\begin{remark}\label{rem:chavez}
To prove our results, we do not need the full strength of this assumption. It suffices to assume \eqref{eq: consecutive zero spacing} or any lower bound of the form 
\begin{equation}
    \gamma'-\gamma \gg  \exp \Bigl({- o\Bigl( \frac{\gamma}{\Log \gamma} \Bigr)} \Bigr).
\end{equation}
In fact, it is sufficient to assume an appropriate effective lower bound on $|{\zeta'(\rho)}|$ to obtain the required lower bound on $\gamma'-\gamma$ (see, e.g., the argument of Titchmarsh \cite[14.31]{titchmarsh1986theory}).
\end{remark}

\subsubsection{The proof}
We first present an elementary fact that we shall use throughout the proof. 
\begin{lemma} \label{lem imp fact}
    Fix $\alpha>0.$ Let $g$ and $h$ be continuous functions on $[\alpha, \infty)$. Assume $h$ is bounded on $[\alpha, \infty)$ and $g(x)=o(1)$ as $x \to \infty$. Then 
    \begin{equation}
        \frac{1}{\beta}\int_\alpha^\beta h(x) g(x) dx=o(1), \qquad \beta \to \infty.
    \end{equation}
\end{lemma}
\begin{proof}[proof of Lemma \ref{lem imp fact}]
    Assume $h$ is bounded by $C$ on $[\alpha, \infty)$. Let $\epsilon>0$ be given. Choose $\beta_\epsilon>\alpha$ large enough so that $|g(x)| \le \epsilon$ for all $x\ge \beta_\epsilon.$  Then for all $\beta \ge \beta_\epsilon,$ we have
    \begin{align}
        \Bigl | \frac{1}{\beta}\int_\alpha^\beta h(x) g(x) dx \Bigr | &= \Bigl | \frac{1}{\beta}\int_\alpha^{\beta_\epsilon} h(x) g(x) dx  + \frac{1}{\beta}\int_{\beta_\epsilon}^\beta h(x) g(x) dx \Bigr |\\
        &\le \frac{ C_\epsilon (\beta_\epsilon-\alpha)}{\beta} + C \epsilon,
    \end{align}
    for some $C_\epsilon \in \R.$ Therefore, 
    \begin{equation}
        \limsup_{\beta \to \infty} \Bigl | \frac{1}{\beta}\int_\alpha^\beta h(x) g(x) dx \Bigr |\le  C \epsilon.
    \end{equation}
\end{proof}

\begin{proof}[proof of Theorem \ref{thm: bias theorem}]
Throughout the proof, assume the notations of Theorem \ref{thm: bias theorem}. To prove the theorem, it will be convenient to first establish the following uniform boundedness

\begin{equation}\label{eq: bound on sum rho}
    \sup_{x\ge 3}\ \bigl | \sum_{\rho}\frac{\Delta_\rho(x)}{x^{1/2} (\Log x)^{-(z+1)}}\bigr |<\infty,
\end{equation}
where $\Delta_{\rho}(x)$ is defined in Theorem \ref{thm: exp form Af}:

\begin{equation}
\frac{\Delta_{\rho}(x)}{x^{1/2}} = - \frac{\sin(\pi z)}{\pi}  x^{i\gamma} \int_{0}^{\beta} e^{-Lu} u^{z}  \mathcal{J}_{\rho}(u;z,w,\mathfrak{f}) du.
\end{equation}
Here $L:=\Log x,  \beta:=1/2-a,$ and $\alpha:=\Re z.$ Hence for every $x\ge 3$, since $e^{-Lu} \le 1$ for $u\ge 0,$ we have the uniform bound
\begin{align}\label{eq:Delta rho uniform bound}
\bigl |\frac{\Delta_{\rho}(x)}{x^{1/2}}\bigr |
&\le \frac{|\sin(\pi z)|}{\pi} \int_{0}^{\beta} e^{-Lu}   u^{\alpha}  \bigl |\mathcal{J}_{\rho}(u;z,w,\mathfrak{f})\bigr | du \\
&\le \frac{|\sin(\pi z)|}{\pi} \Big(\sup_{0\le u\le\beta}|\mathcal{J}_{\rho}(u;z,w,\mathfrak{f})|\Big)\int_0^\beta e^{-Lu}u^{\alpha} du \\
&\le \frac{|\sin(\pi z)|}{\pi} \Big(\sup_{0\le u\le\beta}|\mathcal{J}_{\rho}(u;z,w,\mathfrak{f})|\Big)\Gamma(\alpha+1) L^{-(\alpha+1)}, \label{eq: bound Delta rho .}
\end{align}
where the last inequality follows from the following change of variables $v=Lu$:
\begin{equation}
\int_0^\beta e^{-Lu}u^{\alpha} du
= L^{-(\alpha+1)}\int_0^{\beta L} e^{-v}v^{\alpha} dv
\le L^{-(\alpha+1)}\int_0^\infty e^{-v}v^{\alpha} dv
= \Gamma(\alpha+1)L^{-(\alpha+1)},
\end{equation}
which is valid because $\alpha>-1$.

Let us define the dominating function of $\rho$
\begin{equation} \label{eq: Crho def}
    C_{\rho}(z,w,\mathfrak{f}):= \frac{|\sin(\pi z)|}{\pi} \Big(\sup_{0\le u\le\beta}|\mathcal{J}_{\rho}(u;z,w,\mathfrak{f})|\Big)\Gamma(\Re z+1).
\end{equation}

In particular, a sufficient condition for \eqref{eq: bound on sum rho} is
\begin{equation}\label{eq: sum Crho}
\sum_{\rho} C_{\rho}(z,w,\mathfrak{f}) <\infty,
\end{equation}
since \eqref{eq: bound Delta rho .} and \eqref{eq: sum Crho} implies \eqref{eq: bound on sum rho}.

We will first establish the sufficient condition \eqref{eq: sum Crho} for the uniform boundedness condition \eqref{eq: bound on sum rho}. Throughout this section, write
\begin{equation}
    \alpha:=\Re(z)\in(-1,1), \qquad \beta:=\tfrac12-a\in(0,\tfrac16],
\end{equation}
and for a fixed zero $\rho_0= 1/2 + i \gamma_0$, define
\begin{equation}\label{eq: S rho0 def}
S_{\rho_0}:=\{\gamma\neq\gamma_0: \zeta(1/2 + i \gamma)=0,  |\gamma-\gamma_0|\le 1\},
\end{equation}
and
\begin{equation} \label{eq: Mrho0}
    M_{\rho_0}:=|S_{\rho_0}|\ll \Log |\gamma_0|,
\end{equation}
where the last estimate follows from the Riemann-von Mangoldt formula, and the implied constant is independent of $\gamma_0$.

\begin{prop}\label{lem: spacing implies sum Crho}
The lower bound on consecutive zeros-spacing \eqref{eq: consecutive zero spacing} implies \eqref{eq: sum Crho}, and hence also the uniform boundedness condition \eqref{eq: bound on sum rho}. 
\end{prop}

To prove Proposition \ref{lem: spacing implies sum Crho}, we first need a bound on $C_\rho$ that is defined in \eqref{eq:Delta rho uniform bound}.

\begin{lemma} \label{lem: Crho0 bdd}
There exists $T_0>0$ such that for all $\gamma$ with $|\gamma| \ge T_0$, we have
\begin{equation}
C_{\rho}(z,w,\mathfrak{f})\ \ll_{a,\mathfrak{f}} \exp \Bigl(-\frac{\pi}{4}|\gamma|\Bigr).
\end{equation}
\end{lemma}
\begin{proof}[Proof of Lemma \ref{lem: Crho0 bdd}] Let $\rho_0=\tfrac12+i\gamma_0$ be a nontrivial zero of $\zeta(s)$. By symmetry, we may assume, without loss of generality, that $\gamma_0>0$ and is large.
We first recall that
\begin{equation} 
    C_{\rho_0}(z,w,\mathfrak{f})= \frac{|\sin(\pi z)|}{\pi} \Big(\sup_{0\le u\le\beta}|\mathcal{J}_{\rho_0}(u;z,w,\mathfrak{f})|\Big)\Gamma(\alpha+1) \ll_\mathfrak{f} \sup_{0\le u\le\beta}|\mathcal{J}_{\rho_0}(u;z,w,\mathfrak{f})|,
\end{equation}
where $\mathcal{J}_{\rho_0}$ is defined in \eqref{Jrho}. We bound each factor of $\mathcal{J}_{\rho_0}$. 

To begin with, we derive a bound on $Z_{\rho_0}(\rho_0-u;z).$ Indeed, using the definition of $L_{\rho_0}$ in \eqref{eq: Lrho}, we know that
\begin{equation} 
L_{\rho_0}(s)=\Log(s-1)+\log\zeta(s)-\Log(s-\rho_0).
\end{equation}
Moreover, using \cite[Thm. 9.6(B)]{titchmarsh1986theory} and \eqref{eq: S rho0 def}, we have
\begin{equation}
    \log\zeta(s)-\Log(s-\rho_0)= \sum_{\gamma\in S_{\rho_0}}\Log \bigl(s-\rho\bigr)
+O(\Log(|\gamma_0|)),
\end{equation}
where the RHS is well-defined at $s= \rho_0 -u$ for $u\in(0,\beta]$. Thus, we obtain uniformly for $u\in(0,\beta]$ that
\begin{equation}\label{eq: Lrho sum}
L_{\rho_0}(\rho_0-u) = \Log(\rho_0-u-1) +\sum_{\gamma\in S_{\rho_0}}\Log \bigl(-u+i(\gamma_0-\gamma)\bigr) +O(\Log(|\gamma_0|)).
\end{equation}
Exponentiating \eqref{eq: Lrho sum}, and using $|z|\le 1$, we deduce
\begin{equation}\label{eq: Zrho bdd}
|Z_{\rho_0}(\rho_0-u;z)| =|e^{zL_{\rho_0}(\rho_0-u)}| \ll (|\gamma_0|+3)^{C} e^{C M_{\rho_0}}  |\rho_0-u-1|^{\alpha} \prod_{\gamma\in S_{\rho_0}}\bigl|u+i(\gamma_0-\gamma)\bigr|^{ \alpha},
\end{equation}
for some absolute constant $C>0$. The number $C>0$ may change from line to line if needed. 

We thus need a bound on the product in \eqref{eq: Zrho bdd}. Let us define
\begin{equation}
\eta_{\rho_0}:=
\begin{cases}
\min_{\gamma\in S_{\rho_0}}|\gamma_0-\gamma|, & \text{if }M_{\rho_0}>0,\\
1, & \text{if }M_{\rho_0}=0.
\end{cases}
\end{equation}

\noindent Now for $u\in[0,\beta]$ and $\gamma\in S_{\rho_0}$, we have $|\gamma-\gamma_0|\le 1$, and the triangle inequality implies
\begin{equation}
  \eta_{\rho_0} \le  |\gamma_0-\gamma| \le  |u+i(\gamma_0-\gamma)|\le 2.
\end{equation}
Thus, if $\alpha\ge 0$, using \eqref{eq: Mrho0}, we have for $u\in[0,\beta]$,
\begin{equation}\label{eq: alpha prod}
\prod_{\gamma \in S_{\rho_0}}|u+i(\gamma_0-\gamma)|^{ \alpha}\le 2^{\alpha M_{\rho_0}} \ll e^{C' \cdot \Log |\gamma_0|},
\end{equation}
for some absolute constant $C'>0$.

On the other hand, if $-1<\alpha<0$, using \eqref{eq: Mrho0}, we have for $u\in[0,\beta]$ 
\begin{equation} \label{eq: alpha prod2}
\prod_{\gamma \in S_{\rho_0}}|u+i(\gamma_0-\gamma)|^{ \alpha}\le  \eta_{\rho_0}^{-|\alpha|M_{\rho_0}}.
\end{equation}

However, suppose $\gamma_1 < \gamma_0 < \gamma_2$ are consecutive zeros' ordinates, then using the consecutive spacing bound \eqref{eq: consecutive zero spacing}, we know that 
\begin{equation}\label{eq: eta rho0 lower bound}
\eta_{\rho_0}\ >\  \min \Bigl( \frac{A_1}{|\gamma_0|^{A_1+1}},\frac{A_1}{|\gamma_1|^{A_1+1}} \Bigl)= \frac{A_1}{|\gamma_0|^{A_1+1}},
\end{equation}
for some constant $A_1$ independent of $\gamma_0.$

Therefore, \eqref{eq: Mrho0} and \eqref{eq: eta rho0 lower bound} give
\begin{equation} \label{eq: etarhorM}
\eta_{\rho_0}^{-|\alpha|M_{\rho_0}} =\exp \Bigl(|\alpha|M_{\rho_0}\Log\frac{1}{\eta_{\rho_0}}\Bigr) \le \exp \bigl(C'(\Log(|\gamma_0|))^2\bigr),
\end{equation}
for some possibly larger absolute constant $C'>0$. 

Therefore, using \eqref{eq: Zrho bdd}, \eqref{eq: alpha prod}, \eqref{eq: alpha prod2}, \eqref{eq: etarhorM} and the estimate $|\rho_0-u-1|^{\alpha} \le (|\gamma_0|+3)$, it follows that 

\begin{equation}\label{eq: Zrho bound}
|Z_{\rho_0}(\rho_0-u;z)| \le (|\gamma_0|+3)^{C} e^{C M_{\rho_0}} 
\exp \bigl(C'(\Log(|\gamma_0|))^2\bigr).
\end{equation}

On the other hand, since $\Re(\rho_0 - u)>1/3,$ we have the following uniform bounds on $u\in[0,\beta]$ for the remaining factors in $\mathcal J_{\rho_0}(u;z,w,\mathfrak{f})$:

\begin{equation}
|\zeta(2\rho_0-2u)^{w}|\ll (|\gamma_0|+3)^{\varepsilon} \qquad \text{ for every fixed } \varepsilon>0,
\end{equation}
\begin{equation}
|(\rho_0-1-u)^{-z}|\ll (|\gamma_0|+3)^{C},
\end{equation}
\begin{equation}
|G_{\mathfrak{f}}(\rho_0-u)|\ll_{a,\mathfrak{f}}1,
\end{equation}
and by Stirling's formula,
\begin{equation}\label{eq: Gamma decay}
|\Gamma(\rho_0-u)| \ll |\gamma_0|^{-u} e^{-\pi|\gamma_0|/2} \ll e^{-\pi|\gamma_0|/2}.
\end{equation}
Combining \eqref{eq: Zrho bound}-\eqref{eq: Gamma decay} and using \eqref{eq: Mrho0}, we obtain the following:
\begin{equation}\label{eq: Jrho bound.}
\sup_{0\le u\le\beta}|\mathcal{J}_{\rho_0}(u;z,w,\mathfrak{f})| \ll_{a,\mathfrak{f}} e^{-\pi|\gamma_0|/2}  (|\gamma_0|+3)^{C} \exp \bigl(C'(\Log(|\gamma_0|))^2\bigr).
\end{equation}
Lastly, \eqref{eq: Jrho bound.} is true for any zero $\rho_0$ of $\zeta(s)$, so assuming the ordinate $|\gamma_0|$ is large enough, the sub-exponential term $(|\gamma_0|+3)^{C} \exp \bigl(C'(\Log(|\gamma_0|))^2\bigr)$ is dominated by $e^{\pi|\gamma_0|/4}$. This completes the proof of Lemma \ref{lem: Crho0 bdd}.

\end{proof}

We are now ready to prove Proposition \ref{lem: spacing implies sum Crho}.
\begin{proof}[Proof of Proposition \ref{lem: spacing implies sum Crho}]
    By Lemma \ref{lem: Crho0 bdd}, there exists $T_0$ such that for all $\gamma$ with $|\gamma| \ge T_0$, we have
\begin{equation}
C_{\rho}(z,w,\mathfrak{f})\ \ll\ \exp \Bigl(-\frac{\pi}{4}|\gamma|\Bigr).
\end{equation}

On the other hand, we may sum $C_{\rho}(z,w,\mathfrak{f})$ dyadically:
\begin{equation}
    \sum_\rho C_{\rho}(z,w,\mathfrak{f}) = \sum_{k \ge0} \sum_{\rho\in\mathcal Z_k} C_{\rho}(z,w,\mathfrak{f}),
\end{equation}
where for $k\ge 0$, we define
\begin{equation}
\mathcal Z_k:=\{\rho:\ 2^k\le |\gamma|<2^{k+1}\}.
\end{equation}

By the Riemann-von Mangoldt formula \cite[Ch. 9]{titchmarsh1986theory}, we know $|\mathcal Z_k|\ll 2^k (k+1)$. Thus, 
\begin{equation}
\sum_{\rho\in\mathcal Z_k} C_\rho
\ll |\mathcal Z_k| \exp \Bigl(-\frac{\pi}{4}2^k\Bigr)
\ll 2^k (k+1) \exp \Bigl(-\frac{\pi}{4}2^k\Bigr),
\end{equation}
for all $k\ge k_0$ for some large $k_0$ with $2^{k_0}\ge T_0.$ Lastly, since the series $\sum_k 2^k (k+1)\exp(-(\pi/4)2^k)$ is convergent, and the finitely many zeros with
$|\gamma|\le 2^{k_0}$ contribute a finite amount, it follows that
\begin{equation}
\sum_\rho C_\rho(z,w,\mathfrak{f})<\infty.
\end{equation}
This completes the proof of Proposition \ref{lem: spacing implies sum Crho}.
\end{proof}

We now derive the expansion \eqref{eq: Bf expan} of $B_{\mathfrak{f}}^{\exp}(x;z,w)$. By Theorem \ref{thm: exp form Af}, there exists $T_0>0$ such that for every fixed $x\ge 3$ and every $T\ge T_0$ satisfying \eqref{eq: height T lower bd}, we have
\begin{equation}\label{eq: Af exp exp}
    A_{\mathfrak{f}}^{\exp}(x) - \Delta_1(x)= \Delta_{1/2}(x) + S_T(x) + O \bigl(x^c e^{-T} + x^a\bigr).
\end{equation}

Now letting $T\to\infty$, with the property \eqref{eq: height T lower bd}, in \eqref{eq: Af exp exp} with $x$ fixed, and using Proposition \ref{lem: spacing implies sum Crho}, we obtain 
\begin{equation}\label{eq: Af exp sum}
    A_{\mathfrak{f}}^{\exp}(x) - \Delta_1(x) =  \Delta_{1/2}(x) + \sum_{\rho}\Delta_\rho(x) + O \bigl(x^a\bigr), \qquad x\ge 3.
\end{equation}
The series $\sum_{\rho}\Delta_\rho(x)$ converges absolutely by Proposition \ref{lem: spacing implies sum Crho}.

Dividing \eqref{eq: Af exp sum} by $x^{1/2}L^{w-1}$ and using the definition of $B_{\mathfrak{f}}^{\exp}$ in \eqref{eq: Bf def}, we obtain
\begin{equation}\label{eq: Bf decomp}
    B_{\mathfrak{f}}^{\exp}(x;z,w) = \frac{\Delta_{1/2}(x)}{x^{1/2}L^{w-1}} + L^{-(z+w)}\sum_{\rho} \frac{\Delta_\rho(x)}{x^{1/2}L^{-(z+1)}} + o(1),
\end{equation}
since $a<\tfrac12$. On the other hand, by Theorem \ref{thm: watson exp for Delta}, we have 
\begin{equation}\label{eq: Watson delta half}
  \frac{\Delta_{1/2}(x)}{x^{1/2}(\Log x)^{w-1}} = c_{1/2}(\mathfrak{f};z,w) + o(1), \qquad x\to\infty.
\end{equation}
Combining \eqref{eq: Bf decomp} and \eqref{eq: Watson delta half}, we get the expansion \eqref{eq: Bf expan} of $B_{\mathfrak{f}}^{\exp}$:
\begin{equation} \label{eq: biasexp2}
  B_{\mathfrak{f}}^{\exp}(x;z,w)= c_{1/2}(\mathfrak{f};z,w)+ L^{-(z+w)}\sum_{\rho} \frac{\Delta_\rho(x)}{x^{1/2}L^{-(z+1)}}+ o(1), \qquad x\to\infty.
\end{equation}

We now prove the assertions $(1),  (2),$ and $(3)$ of Theorem \ref{thm: bias theorem}. Throughout the rest of the proof, set 
\begin{equation}\label{eq: S(x) def}
  S(x):= \sum_{\rho} \frac{\Delta_\rho(x)}{x^{1/2}(\Log x)^{-(z+1)}}.
\end{equation}
In the proof of $(2),$ we shall also show that $S(x)$ is not identically zero. This will, in turn, prove the claim that the sum $\sum_{\rho} \Delta_\rho(x)$ is not identically zero.

We begin with the proof of $(1)$.

\begin{proof}[Proof of Theorem \ref{thm: bias theorem}\textup{(1)}]
Assume $\Re(z+w)>0$ and $c_{1/2}(\mathfrak{f};z,w)\neq 0$. By the uniform boundedness condition \eqref{eq: bound on sum rho}, there exists a constant
$C> 0$ such that
\begin{equation}\label{eq: S(x) bdd}
  |S(x)|\le C, \qquad x\ge 3.
\end{equation}

Hence, by the bias expansion \eqref{eq: biasexp2}, we have
\begin{equation}
  \bigl|B_{\mathfrak{f}}^{\exp}(x;z,w)-c_{1/2}(\mathfrak{f};z,w)\bigr| = O\bigl((\Log x)^{-\Re(z+w)}\bigr) + o(1) \xrightarrow[x\to\infty]{} 0,
\end{equation}
since $\Re(z+w)>0$.

By Definition \ref{def: bias }, this means that $A_{\mathfrak{f}}^{\exp}$ has persistent
bias $c_{1/2}(\mathfrak{f};z,w)$ at the scale $x^{1/2}(\Log x)^{w-1}$. This proves $(1)$. 
\end{proof}

We now prove $(2)$ of Theorem \ref{thm: bias theorem}.

\begin{proof}[Proof of Theorem \ref{thm: bias theorem}\textup{(2)}]
Assume that $\Re(z+w)=0$ and $c_{1/2}(\mathfrak{f};z,w)\neq 0$. Throughout this proof, write 
\begin{equation}
    z+w=i\tau,\qquad \tau\in\mathbb R.
\end{equation}

Using \eqref{eq: biasexp2}, \eqref{eq: S(x) bdd}, and $\Re(z+w)=0$, it follows immediately that $B_{\mathfrak{f}}^{\exp}(x;z,w)$ is bounded for all sufficiently large $x$:
\begin{equation}
\bigl|B_{\mathfrak{f}}^{\exp}(x;z,w)\bigr|
\le |c_{1/2}(\mathfrak{f};z,w)| + C + 1 <\infty, \qquad x\to \infty. 
\end{equation}
\vspace{2mm}

We next prove the logarithmic convergence of $B_{\mathfrak{f}}^{\exp}(x;z,w)$ to $c_{1/2}(\mathfrak{f};z,w)$:
\begin{equation}\label{eq: goal log avr}
\lim_{X\to\infty}\frac{1}{\Log X}\int_{e}^{X} \bigl(B_{\mathfrak{f}}^{\exp}(u;z,w)-c_{1/2}(\mathfrak{f};z,w)\bigr) \frac{du}{u}=0,
\end{equation}
where we note that the constant lower limit of integration does not affect the limit value. 

Putting $u=e^t$ and $Y= \Log X$, we see that \eqref{eq: goal log avr} is equivalent to
\begin{equation}
\lim_{Y\to\infty}\frac{1}{Y}\int_{1}^{Y} \bigl(B_{\mathfrak{f}}^{\exp}(e^t;z,w)-c_{1/2}(\mathfrak{f};z,w)\bigr) dt=0.
\end{equation}
Using the expansion \eqref{eq: biasexp2} of $B_\mathfrak{f}$, it is enough to prove that
\begin{equation}\label{eq: goal log S}
\lim_{Y\to\infty}\frac{1}{Y}\int_{1}^{Y} t^{-(z+w)} S(e^t) dt=0,
\end{equation}
where $S(x)$ is given in \eqref{eq: S(x) def}.

\smallskip
For a zero $\rho=\frac{1}{2}+i\gamma$, the Watson expansion \eqref{eq: Delta xi} for $\Delta_{\rho},$ with $M=0,$ gives
\begin{equation}\label{eq:WatsonM0}
\Delta_{\rho}(x)= - \frac{\sin(\pi z)}{\pi} x^{1/2+i\gamma}
\Bigl(\lambda_{\rho,0} \Gamma(1+z) L^{-(1+z)} +O_{\rho}\bigl(L^{-(2+\Re z)}\bigr) \Bigr), \qquad x\to\infty.
\end{equation}
Dividing by $x^{1/2}L^{-(z+1)}$, we get
\begin{equation}\label{eq:normWatson}
N_\rho(x):= \frac{\Delta_\rho(x)}{x^{1/2}L^{-(z+1)}} = a_\rho x^{i\gamma} + o_\rho(1), \qquad x \to \infty,
\end{equation}
where
\begin{equation}\label{eq: arho def}
a_\rho:=-\frac{\sin(\pi z)}{\pi} \Gamma(1+z) \lambda_{\rho,0}.
\end{equation}

By the definition of $\lambda_{\rho,0},$ definition \eqref{eq: Crho def} and $z \not \in \Z$, we have $\lambda_{\rho,0}= \mathcal{J}_\rho(0) \ll_z C_\rho$ where $\mathcal J_\rho$ is given in \eqref{Jrho}. Hence, by the definition of $a_\rho$ \eqref{eq: arho def}, it follows that 
\begin{equation}\label{eq: arho bdd}
    |a_\rho|\ll_z C_\rho.
\end{equation}
Therefore, using Proposition \ref{lem: spacing implies sum Crho}, we know that the series
\begin{equation}\label{eq: f def}
    f(t):=\sum_\rho a_\rho e^{i\gamma t}
\end{equation}
converges absolutely and uniformly in $t\in\R$.

However, using \eqref{eq: bound Delta rho .}, \eqref{eq: arho bdd}, and \eqref{eq: Crho def}, we have
\begin{align}
   |N_\rho(e^t)-a_\rho e^{i\gamma t}| \le |N_\rho(e^t)|+|a_\rho| \ll_z C_\rho,
\end{align}
where the latter is summable by Proposition \ref{lem: spacing implies sum Crho}, so it follows by the dominated convergence theorem and \eqref{eq:normWatson} that 
\begin{equation}\label{eq: S decomp}
S(e^t)=f(t)+o(1), \qquad t\to\infty.
\end{equation}
where $S(e^t),$ $f(t)$ are given in \eqref{eq: S(x) def} and \eqref{eq: f def}.

Substituting \eqref{eq: S decomp} in \eqref{eq: goal log S}, we see that it remains to show 
\begin{equation}\label{eq: goal log f}
\frac{1}{Y}\int_1^Y t^{-i\tau}f(t) dt\to0, \qquad Y \to \infty.
\end{equation}

Recall $z+w=i \tau.$ Expanding \eqref{eq: goal log f} and using absolute convergence of $f(t)$, we have 
\begin{equation}
\frac{1}{Y}\int_1^Y t^{-i\tau}f(t) dt= \sum_\rho a_\rho  \frac{1}{Y} \int_{1}^{Y} t^{-i\tau}e^{i\gamma t} dt.
\end{equation}
Hence it suffices to prove that
\begin{equation}
\frac{1}{Y}\sum_\rho a_\rho \int_{1}^{Y} t^{-i\tau}e^{i\gamma t} dt \longrightarrow 0, \qquad Y\to\infty.
\end{equation}

Indeed, using $|\tau| \le |z| + |w| \le 3$ and integrating by parts give
\begin{align}
\Bigl| \int_{1}^{Y} t^{-i\tau}e^{i\gamma t} dt \Bigr|
&=\Bigl|\left[\frac{t^{-i\tau}e^{i\gamma t}}{i\gamma}\right]_{1}^{Y} +\frac{i \tau}{i\gamma}\int_{1}^{Y} t^{-i\tau-1}e^{i\gamma t} dt \Bigr| \\
&\le \frac{2}{|\gamma|} + \frac{3}{|\gamma|} \int_{1}^{Y} t^{-1} dt  \\
&\ll \frac{\Log Y}{|\gamma|}. \label{eq: onegammabound}
\end{align}
Therefore, from \eqref{eq: onegammabound} and \eqref{eq: arho bdd}, we have
\begin{equation}
\bigl |\frac{1}{Y}\sum_\rho a_\rho \int_{1}^{Y} t^{-i\tau}e^{i\gamma t} dt\bigr | \ll \frac{\Log Y}{Y}\sum_\rho \frac{|a_\rho|}{|\gamma|} \le \frac{\Log Y}{Y}\sum_\rho {|a_\rho|} \xrightarrow[Y\to\infty]{}0 .
\end{equation}
This completes the proof of \eqref{eq: goal log S} and hence \eqref{eq: goal log avr}. It remains to show that the pointwise limit does not exist. Before proving this, we show that there exists a nontrivial zero $\rho_0$ of $\zeta(s)$ such that $a_{\rho_0} \neq 0.$ This will, in turn, imply that $f$ is not identically zero and hence $S$ is not identically zero, as desired.

\begin{lemma} \label{lem: arho0 not zero}
    Assume RH, SZC, and $z \not \in \Z.$ Let $a_\rho$, $f$, and $S$ be as given in \eqref{eq: arho def}, \eqref{eq: f def}, and \eqref{eq: S(x) def}, respectively. Then there exists a nontrivial zero $\rho_0$ of $\zeta(s)$ such that $a_{\rho_0} \neq 0.$ Consequently, $f$, and hence $S,$ is not identically zero.
\end{lemma}

\begin{proof}[proof of Lemma \ref{lem: arho0 not zero}]
    First recall the definitions \eqref{eq: arho def}, $\lambda_{\rho,0}= \mathcal{J}_\rho(0),$ and \eqref{Jrho} and note that for a zero $\rho,$ 
    \begin{equation}\label{eq: arho ext}
        a_\rho=-\frac{\sin(\pi z)}{\pi} \Gamma(1+z) (\rho-1)^{-z} G_{\mathfrak{f}}(\rho)  Z_\rho(\rho;z)  \zeta(2\rho)^{w}  \Gamma(\rho).
    \end{equation}
    Under RH, SZC, and $z \not \in \Z,$ it is clear that all of the factors in \eqref{eq: arho ext} are nonzero except possibly for $G_{\mathfrak{f}}(\rho),$ defined by \eqref{eq:Gf def prod}. We show that there exists a zero $\rho_0$ such that $G_{\mathfrak{f}}(\rho_0) \neq 0.$ 
    
    Indeed, by \eqref{eq:Gf def prod}, $G_{\mathfrak{f}}(s)= \prod_p G_p(s)$, with $G_p$ defined in \eqref{eq: Gp def}, and by the proof of Proposition \ref{prop: properties of Gf} $(b)$, it is clear that $G_{\mathfrak{f}}(\rho)$ converges to a nonzero value iff none of the factors $G_p(\rho)$ is zero.

    Now by \eqref{eq: Gp def}, we observe that if $G_p(s)=0$ for some $s=1/2 +i t,$ then $g_\mathfrak{f}(p^{-s})=0,$ i.e.
    \begin{equation}\label{eq: gf p}
        g_\mathfrak{f}(p^{-1/2} e^{- i t \Log p})=0,
    \end{equation}
    where recall $g_\mathfrak{f}(u)= \sum_{j=0}^\infty \varepsilon_j u^j.$ 
    
    However, $g_\mathfrak{f}(0)=1 \neq 0,$ so by continuity of $g_\mathfrak{f},$ there exists $r>0$ such that $g_\mathfrak{f}(u) \neq0$ for $|u| < r$. Therefore, if \eqref{eq: gf p} is true for some prime $p$ and $t \in \R,$ then necessarily $p \le  r^{-2}.$ In other words, there are finitely many primes $p$ for which \eqref{eq: gf p} can hold. The zeros in \eqref{eq: gf p}, if any, all fall on the circles $|u|=p^{-1/2}$ with $p$ prime, which are subsets of the unit disk.

    However, $g_\mathfrak{f}$ is holomorphic and not identically zero on the unit disk. Thus, for every prime $p\le r^{-2},$ the zeros on the compact circle $|u|=p^{-1/2}$, if any, are finitely many. It remains to account for the fact that a zero on the circle $|u|=p^{-1/2}$ can be hit by infinitely many $t$'s. 
    
    Indeed, for each fixed prime $p\le r^{-2}$ and each fixed zero $u_p$ on $|u|=p^{-1/2}$, the solutions to $p^{-1/2} e^{- i t \Log p}=u_p$ are spaced by $\frac{2 \pi}{\Log p},$ so for every $T>15,$ there are $O_p(T)$ such solutions with $|t|\le T$. Therefore, for every $T>15,$ 
    \begin{equation}\label{eq: Gf OT}
        \Bigl|\{t: |t|\le T , \quad G_{\mathfrak{f}}(1/2 +it)=0\} \Bigr| = O_{\mathfrak{f}}(T) .
    \end{equation} 
    Nevertheless, by the Riemann-von Mangoldt formula, $|\{\rho : \zeta(\rho)=0 , |\Im \rho|\le T\}| \sim \frac{T}{\pi} \Log T$ as $T \to \infty.$ Choosing $T$ large enough and using \eqref{eq: Gf OT}, it follows that there exists a zero $\rho_0$ such that $|\Im\rho_0|\le T$ and $G_{\mathfrak{f}}(\rho_0) \neq 0.$ This proves $a_{\rho_0} \neq 0$.

    Now recall that $f(t)=\sum_\rho a_\rho e^{i\gamma t}$ is absolutely uniformly convergent on $\R.$ Suppose $f$ is identically zero. Then writing $\rho_0= 1/2 +i \gamma_0$ and using the dominated convergence theorem, we obtain
    \begin{align}
        0 &= \lim_{Y \to \infty} \frac{1}{Y} \int_0^Y f(t) e^{- i \gamma_0 t} dt= \lim_{Y \to \infty} \sum_\rho a_\rho  \frac{1}{Y} \int_0^Y e^{ i (\gamma- \gamma_0) t} dt  \label{eq: 246} \\ 
        &=\sum_\rho a_\rho \lim_{Y \to \infty}  \frac{1}{Y} \int_0^Y e^{ i (\gamma- \gamma_0) t} dt= a_{\rho_0}, \label{eq: 247}
    \end{align}
    contradicting $a_{\rho_0} \neq 0.$ This proves that $f$ is not identically zero.

    Finally, suppose that $S$ is identically zero, then by \eqref{eq: S decomp}, it follows that $f(t) \to 0$ as $t\to \infty.$ The same contradiction follows from \eqref{eq: 246} and \eqref{eq: 247}, since $f$ is bounded. This completes the proof of Lemma \ref{lem: arho0 not zero}
\end{proof}

We now show that the pointwise limit does not exist. Combining \eqref{eq: biasexp2} and \eqref{eq: S decomp} yields
\begin{equation}
B_{\mathfrak{f}}^{\exp}(e^t;z,w)=c_{1/2}(\mathfrak{f};z,w) + t^{-i\tau}f(t) +o(1), \qquad t\to\infty.
\end{equation}

Thus, to prove the pointwise limit does not exist, it suffices to show the limit as $t \to \infty$ of $t^{-i\tau}f(t)$ does not exist. Indeed, suppose, for a contradiction, that $t^{-i\tau}f(t) \to \ell \in \C$ as $t \to \infty$. Then
\begin{equation} \label{eq: f dne}
    f(t)= \ell  t^{i\tau}  + o(1), \qquad t \to \infty.
\end{equation}
Now by Lemma \ref{lem: arho0 not zero}, there exists a zero $\rho_0$ such that $a_{\rho_0} \neq 0.$ By absolute convergence of $f(t),$ we know that
\begin{equation}\label{eq: arho0}
    a_{\rho_0}= \lim_{Y\to\infty}\frac{1}{Y}\int_1^Y f(t)e^{-i\gamma_0t} dt.
\end{equation}
However, substituting \eqref{eq: f dne} into \eqref{eq: arho0} and using Lemma \ref{lem imp fact}, we see that 
\begin{equation}
    a_{\rho_0}=\ell\lim_{Y\to\infty}\frac{1}{Y}\int_1^Y t^{i\tau}e^{-i\gamma_0t} dt=0,
\end{equation}
where the last equality follows by an integration by parts. This is a contradiction. Therefore,  the limit as $t \to \infty$ of $t^{-i\tau}f(t)$ does not exist.

This proves $A_{\mathfrak{f}}^{\exp}$ has apparent bias $c_{1/2}(\mathfrak{f};z,w)$ and finishes the proof of Theorem \ref{thm: bias theorem}$(2)$.

\end{proof}

Finally, we prove Theorem \ref{thm: bias theorem}$(3)$.

\begin{proof}[Proof of Theorem \ref{thm: bias theorem}\textup{(3)}]
First, if $c_{1/2}(\mathfrak{f};z,w)=0$ and $\Re(z+w) > 0,$ it is immediate from \eqref{eq: biasexp2} that the pointwise limit of $B_{\mathfrak{f}}^{\exp}(x;z,w)$ is 0, so there is no persistent bias, and since the limit exists, there is no apparent bias. Likewise, if $c_{1/2}(\mathfrak{f};z,w)=0$ and $\Re(z+w) = 0,$ the same arguments in the proof of Theorem \ref{thm: bias theorem}$(2)$ show that the limit of $B_{\mathfrak{f}}^{\exp}(x;z,w)$ does not exist and the logarithmic average tends to $0,$ so by definition it follows that there is no persistent or apparent bias.

It remains to show that $B_{\mathfrak{f}}^{\exp}(x;z,w)$ is unbounded when $\Re(z+w) < 0.$ Indeed, since $f\not\equiv 0$ (recall the definition of $f$ in \eqref{eq: f def}), we may choose $t_0\in\R$ with $|f(t_0)|>0$ and set $\delta:=\frac{|f(t_0)|}{4}>0.$ Now since $f$ is Bohr almost periodic (a uniform limit of trigonometric polynomials), the set of $\delta$-almost periods
\begin{equation}
     \mathcal{T}_\delta:=\Bigl\{s\in\R:\ \sup_{t\in\R}|f(t+s)-f(t)|<\delta\Bigr\}
\end{equation}
is relatively dense: there exists $M>0$ such that $[t,t+M]  \cap   \mathcal{T}_\delta \neq \emptyset$ for all $t \in \R$ \cite[p. 729]{oliaro2012almost}.

Thus, for $s\in\mathcal{T}_\delta$, we have
\begin{equation}
    |f(t_0+s)|
\ge |f(t_0)|-|f(t_0+s)-f(t_0)|
> |f(t_0)|-\delta
=3\delta
>2\delta.
\end{equation}
Thus $t_0+\mathcal{T}_\delta \subseteq \{t:\ |f(t)|>2\delta\}$, and since
$\mathcal{T}_\delta$ is relatively dense, so is $\{t:\ |f(t)|>2\delta\}$.

Therefore, there exists a sequence $(t_n)_{n\ge1}$ such that $t_n\to\infty$ and
\begin{equation}
|f(t_n)|\ge 2\delta \qquad n\ge 1.    
\end{equation}

Using \eqref{eq: S decomp}, we get for large $n \ge 1$ that
\begin{equation}\label{eq: S(et)}
  |S(e^{t_n})|\ge \delta.    
\end{equation}

Finally, put $x_n:=e^{t_n}.$ Substituting $x=x_n$ into \eqref{eq: biasexp2}, applying \eqref{eq: S(et)}, and using $\Re(z+w)<0$ and $t_n\to\infty$, we get
\begin{equation}
      |B_{\mathfrak{f}}^{\exp}(x_n;z,w)|
  \ge t_n^{-\Re(z+w)} |S(x_n)| - |c_{1/2}(\mathfrak{f};z,w)| - 1
  \ge \delta t_n^{-\Re(z+w)} - |c_{1/2}(\mathfrak{f};z,w)| - 1
  \xrightarrow[n\to\infty]{}\infty,
\end{equation}
which proves that $B_{\mathfrak{f}}^{\exp}(x;z,w)$ is unbounded as $x\to\infty$.

\end{proof}

This completes the proof of Theorem \ref{thm: bias theorem}.

\end{proof}

\section{Conclusion}

In this paper, we have shown that the type of bias at the scale $x^{1/2}(\Log x)^{w-1}$ is completely determined by the parameters $(\varepsilon_1,\varepsilon_2)$, while the higher coefficients $(\varepsilon_k)_{k\ge 3}$ contribute only at a strictly lower level in the explicit formula. The methods of this paper suggest several further directions, including extensions to the complementary parameter ranges not discussed here. We plan to address these questions in future work.

\section*{Acknowledgements}

I am grateful to Ghaith Hiary for his guidance and support throughout the process of developing this work. I am also grateful to Ovidiu Costin for several helpful discussions in asymptotic analysis. I further thank Greg Martin and Timothy S. Trudgian for answering some questions concerning their work on fake M\"obius functions.

\bibliographystyle{plain}
\bibliography{ref}

\bigskip

\noindent\textsc{Ali Saraeb}\\
Department of Mathematics\\
The Ohio State University\\
Columbus, OH 43210, USA\\
\textit{Email:} \texttt{saraeb.1@osu.edu}

\end{document}